\documentclass[12pt]{amsart}
\usepackage{amssymb,amsmath,amsthm}
\makeatletter\@addtoreset{equation}{section} \makeatother

\newcommand{\nad}[2]{\genfrac{}{}{0pt}{}{#1}{#2}}
\newcommand{\p}{\partial}

\newcommand{\mref}[1]{(\ref{#1})}
\renewcommand{\a}{\alpha}
\def\beq#1#2\eeq{%
        \begin{equation}%
        \label{#1}%
            #2%
        \end{equation}%
    }

\newtheorem{theorem}[equation]{Theorem}
\newtheorem{proposition}[equation]{Proposition}
\newtheorem{lemma}[equation]{Lemma}
\newtheorem{corollary}[equation]{Corollary}

\theoremstyle{remark}
\newtheorem{remark}[equation]{Remark}

\theoremstyle{definition}
\newtheorem{example}[equation]{Example}
\newtheorem{definition}[equation]{Definition}
\newtheorem{convention}[equation]{Convention}

\DeclareMathAlphabet{\mathbbold}{U}{bbold}{m}{n}
\def \k {\mathbbold{k}}

\def \A {\mathbb{A}}
\def \Q {\mathbb{Q}}
\def \C {\mathbb{C}}
\def \R {\mathbb{R}}
\def \Z {\mathbb{Z}}
\def \SS {\mathbb{S}}
\def \Sy {\mathrm{S}}
\def \LL {\mathcal{L}}
\def \FF {\mathcal{F}}
\def \K {\mathcal{K}}
\def \Sympl {\mathfrak{S}}
\def \codim {\mathrm{codim}}
\def \mult  {\mathrm{mult}}
\def \supp  {\mathrm{supp}}
\def \rk  {\mathrm{rk}}
\def \Res  {\mathrm{Res}}
\def \ii  {\mathrm{i}}

\def \lb {\upshape{(}}
\def \rb {\upshape{)}}
\def \blfloor {\Big\lfloor}
\def \brfloor {\Big\rfloor}

\def \le {\leqslant}
\def \ge {\geqslant}

\def \nlb {\nolinebreak}

\title{On unitary submodules in the polynomial representations of rational Cherednik algebras}
\author{M.\,Feigin, C.\,Shramov}

\begin{document}

\begin{abstract}
We consider representations of rational Cherednik algebras which are particular ideals in the ring of
polynomials. We investigate convergence of the integrals which express the Gaussian inner product on
these representations. We derive that the integrals converge for the minimal
submodules in types $B$ and $D$  for the
singular values suggested by Cherednik with at most one exception, hence the corresponding modules are unitary. The analogous result on unitarity of the minimal submodules in type $A$ was obtained by Etingof and Stoica, we give a different
proof of convergence of the Gaussian product in this case. We also obtain partial
results on unitarity of the minimal submodule in the case of exceptional Coxeter groups
and group $B$ with unequal parameters.
\end{abstract}

\address{School of Mathematics and Statistics, University of Glasgow, 15 University Gardens,
Glasgow G12 8QW, UK}

\email{misha.feigin@glasgow.ac.uk}

\address{Steklov Mathematical Institute, Gubkina str., 8, Moscow, 119991, Russia}

\email{shramov@mccme.ru}

\maketitle
\tableofcontents

\section{Introduction}
\label{section:intro}

Let ${\mathcal{R}} \subset \R^N$ be an irreducible Coxeter root system, let
$W$ be the corresponding Coxeter group which is generated by
orthogonal reflections $s_{\a}$ with respect to the hyperplanes
$(\a,x)=0$ where $\a \in \mathcal{R}$, $x=(x_1,\ldots,x_N)$ and
$(\cdot, \cdot)$ denotes the standard inner product in $\R^N$ (see
\cite{Hum}). Let $c: \mathcal{R} \to \R$ be a $W$-invariant
function. The corresponding rational Cherednik algebra $H_c(W)$ (see
\cite{EG}) is generated by the group algebra $\C W$ and two
commutative polynomial subalgebras $\C[x]=\C[x_1,\ldots,x_N]$,
$\C[y_1,\ldots, y_N]$. The algebra can be defined by its faithful
representation $\phi$ in the space of polynomials $\C[x]$. In this
representation $\phi|_{\C W}$ is the reflection representation of the
group algebra $\C W$, $\phi(p(x))$ is the operator of multiplication by
$p(x)$, and $\phi(p(y_1,\ldots,y_N))$ is the operator $p(\nabla_1,
\ldots, \nabla_N)$ where $\nabla_i$ are (commuting) Dunkl operators
\cite{D} corresponding to the basis vectors $\xi=e_i$: \beq{Dunkl}
\nabla_\xi = \partial_\xi- \sum_{\a \in \mathcal{R}_+} \frac{c(\a)
(\a, \xi)}{(\a,x)}(1-s_\a), \eeq where $\mathcal{R_+}$ is the set of
positive roots.

The study of unitary representations of the algebra $H_c(W)$ was
initiated in the paper by Etingof, Stoica, Griffeth \cite{ESG} (see
also \cite{Chered}). Recall that category $\mathcal O$ consists of
finitely generated modules such that all Dunkl operators act locally
nilpotently \cite{DO}. The simple objects $L_\tau$ in category
$\mathcal O$ are parametrized by the irreducible modules $\tau$ for
the corresponding Coxeter group $W$. The module $L_\tau$ carries a
$W$-invariant nondegenerate Hermitian form $(\cdot,\cdot)_\tau$ satisfying
$$
 (x_i u, v)_\tau = (u, y_i v)_\tau
$$
for any $u,v \in L_\tau$, for any $i=1,\ldots,N$.
This form is  unique up to
proportionality. The unitary modules are such that this form can be
scaled to be positive definite.

Of particular interest there is Cherednik's question on unitarity of
the minimal submodule $\SS_c$ in the polynomial representation
$\C[x]$ (see \cite[Section 4.6]{ESG} and~\cite{Chered}). This submodule has the form $\SS_c  \cong L_{\tau_c}$ where $\tau_c$ is an irreducible $W$-module which might depend on $c$.
Submodule $\SS_c$ is unique and
it is non-trivial only for the so-called singular multiplicities $c$
when the polynomial representation is reducible. The singular
multiplicities were completely determined in \cite{DJO}. In the case
of constant multiplicity function they are special rational numbers
with the denominators $d_i$ which are degrees of the corresponding
Coxeter group. Cherednik's question is whether the minimal submodule
$\SS_c$ is unitary when $c=1/d_i$.

 It is shown in \cite[Proposition 4.12]{ESG} (see also \cite{Chered}) that
unitarity of the minimal submodule follows from the convergence of
the integral
\begin{equation}\label{Gaussian}
\gamma_c(f) = \int_{\R^N} {|f(x)|^2 e^{-\frac12 |x|^2}}{\prod_{\a \in \mathcal{R}_+}
|(\a,x)|^{-2 c(\a)}} dx
\end{equation}
for all $f \in \SS_c$.
This is due to the observation that
$$
(f, f)_{\tau_c} = \lambda \gamma_c(e^{-\frac12 \sum_{i=1}^N \nabla_i^2} f)
$$
for some constant $\lambda \in \R$ independent of $f\in \SS_c$, and to the obvious inequality $\gamma_c(f) \ge 0$.
Thus the related question posed in \cite{Chered, ESG} is on convergence of the integral \mref{Gaussian} which in such case is called the Gaussian inner product.
It is shown in \cite[Theorem 5.14]{ESG} that this integral does
converge in the case ${\mathcal R}=A_{N-1}$ hence the questions
have positive answer in this case.

In this paper we show unitarity of the minimal submodules in the
polynomial representations for the algebras $H_c(W)$ in certain cases by establishing the
convergence of the above integral, in particular we give another proof of
convergence for the $A_{N-1}$ case (c.f. suggestions in \cite{Chered}). More exactly we show that
$$\Phi_f=|f(x)|\prod_{\a \in \mathcal{R}_+}|(\a,x)|^{-c(\a)}$$ is
locally $L^2$-integrable in $\R^N$ for any $f\in M$ where $M$ is an appropriate ideal. This implies, in particular, that the Cherednik's question
has positive answer in types $B$ and $D$ except for the singular
value $1/N$ in the case of $D_N$ with odd $N$ (see Theorem
\ref{mainth} which is our main result). In the latter case the
answer actually happens to be negative (\cite{Grpc}; see Proposition~\ref{prgrif} below).

The structure of the paper is as follows. In Section~\ref{section:ideals} we consider special ideals in the ring $\C[x]$ and find their generators which are singular polynomials for the corresponding rational Cherednik
algebras of type $H_c(G(m,p,N))$. In Section~\ref{section:L2} we recall the
algebro-geometric technique of checking local integrability and
apply it to our situation by producing an explicit log resolution of
the hyperplane arrangement corresponding to the poles of $\Phi_f$.
The explicit estimates for particular cases are gathered in
Section~\ref{section:brute-force}. In
Section~\ref{section:conclusions} we complete the proof of
convergence of integrals \mref{Gaussian} for $A, B, D$ cases and deduce unitarity of the corresponding
minimal representations. In Section \ref{section:moreunitarity} we present a few
results on the convergence of the Gaussian product \mref{Gaussian} mainly for the case of exceptional Coxeter groups (see
Propositions~\ref{next-to-h-singular-value}, \ref{proposition:other-exceptional-convergence}, and also Proposition \ref{Res}).
In the last section we discuss a few examples when the
minimal submodule is not unitary or when at least the integral \mref{Gaussian} is not convergent on the minimal submodule.

\section{$\pmb H_c$-invariant  ideals}
\label{section:ideals}

In this section we discuss special ideals in the polynomial ring
$\C[x]=\C[x_1,\ldots,x_N]$ which are 
invariant under certain appropriate rational 
Cherednik algebra. 
We specify singular polynomials generating these representations.

Let $\Delta(x_1, \ldots, x_p)$ be the
Vandermonde determinant, that is
$$\Delta(x_1, \ldots, x_p)=
\prod_{i<j}^p (x_i-x_j)$$
for $2\le p\le N$ and $\Delta(x_1)=1$.

Let $\nu=(\nu_1, \ldots, \nu_l)$ be a partition of $N$, that is
$\nu_i \ge \nu_{i+1}$, $\nu_i \in \Z_+$ and $\sum \nu_i = N$. Let
$l(\nu)=l$ be the length of the partition. Define the associated
polynomial
\begin{multline}\label{pnu}
p_\nu(x) =\\ =\Delta(x_1,\ldots,x_{\nu_1})\cdot
\Delta(x_{\nu_1+1},\ldots,x_{\nu_1+\nu_2})\cdot\ldots \cdot
\Delta(x_{\nu_1+\ldots+\nu_{l-1}+1},\ldots,x_{N}).
\end{multline}

Let $k$ be an integer, $1 \le k < N$. Consider the ideal $I_k$ in
the ring $\C[x_1,\ldots, x_N]$ consisting of polynomials $p(x)$ such
that $p(x)=0$ whenever $x_{i_1}=x_{i_2}=\ldots=x_{i_{k+1}}$ for some
indexes $1\le i_1<i_2<\ldots<i_{k+1} \le N$.

It is clear that the image of $p_\nu(x)$ under any $\sigma\in \Sy_N$
is contained in $I_k$
if the length  $l(\nu)\le k$. Moreover, the following proposition is
contained in \cite{ESG}.

\begin{proposition}[{\cite[Section 5.3]{ESG}}]\label{pr1}
Let $N=kq+s$ where $q,s \in \Z_{\ge 0}$, $s< k$. Let $\nu_N^k$ be the partition
$\nu_N^k=((q+1)^{s}, q^{k-s})$. Then the ideal $I_k$ in the ring
$\C[x]$ is generated by the $\Sy_N$-images of the polynomial
$p_{\nu_N^k}(x)$.
\end{proposition}

Indeed, it is shown in \cite[Theorem 5.10]{ESG} that $I_k$ is an
irreducible module over the rational Cherednik algebra $H_c(\Sy_N)$ with the
parameter $c=1/(k+1)$. Therefore it has to be generated as ideal in
$\C[x]$ by its lowest homogeneous component.
It is determined in
\cite{ESG} (see the proof of Proposition~5.16) that the lowest
homogeneous component of the module $I_k$ is linearly generated by
the $\Sy_N$-orbit of $p_{\nu_N^k}(x)$ (under the geometric action of
$\Sy_N$ in $\C[x]$).

\medskip
Consider now the ideal $I_k^\pm$ in $\C[x]$
which consists of the polynomials vanishing on the union of planes
\begin{equation}\label{epspl}
\varepsilon_{i_1} x_{i_1}=\varepsilon_{i_2} x_{i_2}=\ldots=\varepsilon_{i_{k+1}} x_{i_{k+1}}
\end{equation}
where $\varepsilon_{i_s} = \pm 1$ and the indexes $1\le
i_1<i_2<\ldots<i_{k+1} \le N$. This ideal is a module over the
rational Cherednik algebra $H_c(D_N)$ with the parameter $c=1/(k+1)$
\cite[Section 4.3]{F}. It is also a module over $H_c(B_N)$ with the
parameters $c(e_i \pm e_j)=1/(k+1)$, $c(e_i)$ is arbitrary
(see~\cite[Section 4.2]{F}).

\begin{proposition}\label{genidi}
The ideal $I_k^\pm \subset \C[x]$ is generated by the $\Sy_N$-images
of the polynomial $p_{\nu_N^k}(x_1^2, \ldots, x_N^2)$.
\end{proposition}

We actually prove the following slightly more general result.
\begin{proposition}\label{propsingM}
Let $m \ge 2$ be an integer. Consider the ideal $I_k^{(m)} \subset
\C[x]$ consisting of polynomials vanishing on the planes
\mref{epspl}, where $\varepsilon_{i_s}^m =1$ and the indexes $1\le
i_1<i_2<\ldots<i_{k+1} \le N$. This ideal is generated by the
$\Sy_N$-images of the polynomial $p_{\nu_N^k}(x_1^m, \ldots,
x_N^m)$.
\end{proposition}
\begin{proof}
Consider first the lowest homogeneous component $M \subset
I_k^{(m)}$, and let $q \in M$. Let $E_i$ for $1\le i \le N$ be the
idempotents
$$
E_i=\frac{1}{m} \sum_{k=0}^{m-1} s_i^k,
$$
where $s_i$ multiplies the basis vector $e_i$ by $\xi=e^{2\pi\ii/m}$, while
$s_i(e_l)=e_l$ for $l \ne i$. Consider the difference
\beq{reflsi}
q(x) - E_i q(x) = x_i r_i(x), \,\,\,\,\, i=1,\ldots,N,
\eeq
where
$r_i(x)$ are some polynomials. The collection of planes \mref{epspl}
is invariant with respect to reflections $s_i$ and therefore
$q(x)-E_i q(x) \in I_k^{(m)}$. Since $x_i$ is not identically zero
on the planes from \mref{epspl} we conclude that $r_i(x) \in
I_k^{(m)}$. By minimality of the degree of $q(x)$ it follows that
$q(x)=E_i q(x)$, and therefore $s_i q(x)=q(x)$. Thus
$$
q(x)= \tilde q(y_1, \ldots, y_N),
$$
where $y_i=x_i^m$, $i=1,\ldots,N$, and $\tilde q$ is a polynomial.
Now $\tilde q \in I_k \subset \C[y]$, and therefore by Proposition
\ref{pr1} the polynomial $q(x)$ is a linear combination of the
$\Sy_N$-images of the polynomial $p_{\nu_N^k}(x_1^m, \ldots,
x_N^m
)$.

The rest of the Proposition follows by induction on the degree of a
polynomial $q(x) \in I_k^{(m)}$. Indeed we again apply the relations
\mref{reflsi}. They imply by induction that $ q(x)-\hat q(x) $ has
the required form where $\hat q(x)=\prod_{i=1}^N E_i q(x)$. Since
$s_i \hat q(x) = \hat q(x)$, we have $\hat q(x) \in I_k \subset
\C[y]$. So $\hat q$ has the required form, and hence the statement
for $q(x)$ also follows.
\end{proof}

Recall that the complex reflection group $G(m,p,N)$ is defined when
$p|m$, it is generated by the elements $s_{ij}^k$ for $1\le i <j \le N$,
$k=0,\ldots,m-1$, and the elements $\tau_i$ for $i=1,\ldots,N$. The
element $\tau_i$ acts on the basis coordinate functions as $\tau_i(x_i)=\eta
x_i$, where $\eta=e^{2\pi\ii  p/m}$ 
and $\tau_i(x_j)=\nlb x_j$ for $j\ne i$. The elements $s_{ij}^k$ defined
for $i \ne j$  act as $s_{ij}^k (x_j)=\xi^{-k} x_i$, $s_{ij}^k
(x_i)=\xi^{k} x_j$, where $\xi=e^{2\pi\ii/m}$, and $s_{ij}^k
(x_l)=x_l$ for $l\ne i,j$.

It follows from \cite[Section 7]{F} that the ideal $I_k^{(m)}$ is a module over the
rational Cherednik algebra $H_c(G(m,p,N))$ when $c_1=1/(k+1)$, therefore Proposition \ref{propsingM} has the following corollary.

\begin{corollary}\label{corsing}
The polynomials $p_{\nu_N^k}(x_1^m, \ldots, x_N^m)$, $1 \le k \le
N-1$, are singular polynomials for the rational Cherednik algebra
$H_c(G(m,p,N))$. More exactly,
$$\nabla_i p_{\nu_N^k}(x_1^m, \ldots,
x_N^m)=0$$ for all $i=1,\ldots, N$, where $\nabla_i$ is the
Dunkl--Opdam operator \lb see \cite{DO}\rb
\begin{equation}\label{cdunkl} \nabla_i = \partial_i-c_1 \sum_{\nad{j=1}{j\ne i}}^N \sum_{k=0}^{m-1}
\frac{1-s_{ij}^k}{x_i-\xi^k x_j}-\sum_{t=1}^{\frac{m}{p}-1} c_{t+1}
\sum_{s=0}^{\frac{m}{p}-1}\frac{\eta^{-st}\tau_i^{s}}{x_i} \end{equation}
with $c_1=1/(k+1)$.
\end{corollary}

Define now the ideal $J_k = J_k^N \subset \C[x_1,\ldots, x_N]$, $0
\le k \le N-1$, which consists of the polynomials vanishing on the
union of planes
$$
x_{i_1}=\ldots=x_{i_{k+1}}=0
$$
for arbitrary indexes $1 \le i_1 <\ldots < i_{k+1} \le N$.

\begin{proposition}\label{Jgen}
The ideal $J_k \subset \C[x]$ is generated by the $\Sy_N$-images of
the polynomial $x_1\cdot\ldots\cdot x_{N-k}$.
\end{proposition}
\begin{proof}
Let $f$ be an element from $J_k$. Consider the Taylor expansion with
respect to the variable $x_N$:
$$
f=\sum_{i=0}^{\text{deg} f} x_N^i g_i(x_1,\ldots,x_{N-1}).
$$
The polynomials $g_i$ then have to satisfy $g_0 \in J_{k-1}^{N-1}$, $g_i \in J_k^{N-1}$ for $i>0$. The statement follows by induction on the dimension.
\end{proof}

The ideal $J_k$, $0\le k \le N-1$,
is a representation of the rational Cherednik algebra $H_c(G(m,p,N))$
if (and only if) multiplicities satisfy the relation
$k c_1 + p^{-1}c_2=m^{-1}$ (where in $p=m$ case
one assumes $c_i=0$ for $i\ge 2$) by~\cite[Proposition 9]{F}.
In particular the ideal $J_k$ for $1\le k \le N-1$ is a module over the rational Cherednik
algebra $H_c(D_N)$ with $c=\frac{1}{2k}$ \cite[Section 4.3]{F}. Also for any $0\le k \le N-1$ the ideal $J_k$ is a module over $H_c(B_N)$ if
the parameters satisfy the relation $2kc_1+2c_2=1$ where
$c_1 = c(e_i \pm e_j)$ and $c_2 = c(e_i)$ \cite[Section 4.2]{F}.

Proposition \ref{Jgen} has the following corollary.

\begin{corollary}\label{corollary:Jgen}
The polynomials $x_1\cdot\ldots\cdot x_k$, $1 \le k \le N$,
are singular with respect to $H_c(G(m,p,N))$.
More exactly,
$$ \nabla_i (x_1\cdot\ldots\cdot x_k) =0$$
for all $i=1, \ldots, N$, where $\nabla_i$
is the Dunkl--Opdam operator \mref{cdunkl}
and the multiplicities satisfy $m(N-k)c_1 + m p^{-1} c_2=1$.
\end{corollary}

\begin{remark}
Corollary~\ref{corollary:Jgen}
for $k=1$ is contained in~\cite[Proposition~4.1]{CE}
where it is generalized in a different direction.
\end{remark}

We are going to construct some more singular polynomials for the rational Cherednik
algebra $H_c(G(m,p,N))$. Firstly we need the following lemma.

\begin{lemma}\label{lemalt}
Let $L$ be the operator
$$
L= \sum_{j=2}^{n+1} \sum_{k=0}^{m-1}
\frac{1-s_{1j}^k}{x_1-\xi^k x_j}.
$$
Then
\beq{altlem1}
L \left(x_1^{m k} \Delta(x_2^m, \ldots, x_{n+1}^m)\right)= \p_{x_1} \left(x_1^{m k} \Delta(x_2^m, \ldots, x_{n+1}^m)\right),
\eeq
for $0 \le k \le n$,
\beq{altlem1shtrih}
L \left(x_1^{m k+1} \Delta(x_2^m, \ldots, x_{n+1}^m)\right)= (\p_{x_1} + \frac{m-1}{x_1}) \left(x_1^{m k+1} \Delta(x_2^m, \ldots, x_{n+1}^m)\right),
\eeq
for $0 \le k \le n-1$,
\beq{altlem2}
L \left(x_1^{m k} \Delta(x_2^m, \ldots, x_{n+1}^m)\prod_{j=2}^{n+1} x_j \right)= \p_{x_1} \left(x_1^{m k} \Delta(x_2^m, \ldots, x_{n+1}^m)\prod_{j=2}^{n+1} x_j\right),
\eeq
for $0 \le k \le n$, and
\begin{multline}\label{altlem2shtrih}
L \left(x_1^{m k+1} \Delta(x_2^m, \ldots, x_{n+1}^m)
\prod_{j=2}^{n+1} x_j \right)=\\
= (\p_{x_1}-\frac{1}{x_1})
\left(x_1^{m k+1} \Delta(x_2^m, \ldots, x_{n+1}^m)\prod_{j=2}^{n+1} x_j\right),
\end{multline}
for $0 \le k \le n$.
\end{lemma}
\begin{proof}
We rewrite Vandermonde determinant using anti-symmet\-ri\-za\-ti\-on 
with respect
to the group $\mathrm{S}_n$ acting by permutations of the
variables $x_2,\ldots,x_{n+1}$:
$$
\Delta(x_2^m,\ldots, x_{n+1}^m)= Alt \prod_{j=1}^{n} x_{j+1}^{m(j-1)},
$$
where $Alt=\sum_{g \in \mathrm{S}_n} sign(g) g$. Note that the operator
$L$ is $G(m,1,n)$-invariant, where the group $G(m,1,n)$ is generated
by $s_{ij}^k$, $2 \le i < j \le n+1$ and $\tau_i$, $2 \le i \le n+1$.
Therefore
\begin{multline*}
L \left(x_1^{mk} \Delta(x_2^m, \ldots, x_{n+1}^m)\right) =
Alt \Big(L \big(x_1^{mk} \prod_{j=1}^{n} x_{j+1}^{m(j-1)}\big)\Big)
=\\
=\sum_{i=2}^{n+1} Alt\Big( \sum_{k=0}^{m-1}
\frac{1-s_{ij}^k}{x_i-\xi^k x_j} x_1^{m k} \prod_{j=1}^{n} x_{j+1}^{m(j-1)}\Big)
\end{multline*}
and the right-hand side is polynomial in $x_2^m, \ldots, x_{n+1}^m$. Now
\begin{multline*}
Alt\Big(\sum_{k=0}^{m-1}
\frac{1-s_{ij}^k}{x_i-\xi^k x_j} x_1^{mk}
\prod_{j=1}^{n} x_{j+1}^{m(j-1)}\Big) =\\
=  \left\{\begin{array}{cc}
m x_1^{mk-1} \Delta(x_2^m,\ldots,x_{n+1}^m),  &  2 \le i \le k+1,\\
0, &  k+2 \le i \le n+1,
\end{array}
\right.
\end{multline*}
hence the statement \mref{altlem1} follows. Similarly
\begin{multline*}
Alt\Big(\sum_{k=0}^{m-1}
\frac{1-s_{ij}^k}{x_i-\xi^k x_j}
x_1^{mk+1}\prod_{j=1}^{n} x_{j+1}^{m(j-1)}\Big) =\\
=  \left\{\begin{array}{cc}
m x_1^{mk} \Delta(x_2^m,\ldots,x_{n+1}^m),  &  2 \le i \le k+2,\\
0, &  k+3 \le i \le n+1,
\end{array}
\right.
\end{multline*}
hence the statement \mref{altlem1shtrih} holds.
The statements \mref{altlem2}, \mref{altlem2shtrih} follow analogously.
\end{proof}

\begin{proposition}\label{bnsingprop}
Let $N= \sum_{i=1}^{r-1} \nu_i$ where $\nu_i \in \Z_{>0}$
and $|\nu_i - \nu_j| \in \{0,1\}$ for $1 \le i,j \le r-1$.
Denote
$$
{\mathcal I}_i = \{m \in \Z_+ |
\sum_{j=1}^{i-1} \nu_j +1 \le m \le \sum_{j=1}^{i} \nu_j \}.
$$
Let $T \subset \{1,\ldots,N\}$ be a subset of indexes
of size $|T|=r-s-1$ for some $0\le s \le r-1$.
Let $N= (r-1)q + t$ with $0 \le t <r-1$ so that $\nu_i = q$ or $\nu_i=q+1$.
Assume that if there exists $i \in T$ such
that $\nu_i=q+1$ then for all $j$ such that $\nu_j=q$ one has $j \in T$.
Then the polynomial
$$
p_{\nu,T}^{(m)}=p_{\nu}(x_1^m,\ldots,x_N^m) \prod_{i\in T} \prod_{j \in {\mathcal I}_i} x_j
$$
where $p_{\nu}$ is defined by~\mref{pnu}, is $G(m,p,N)$-singular.
More exactly one has
\beq{dungmp}
\nabla_i p_{\nu, T}^{(m)} = 0
\eeq
for $1 \le i \le N$,
where $\nabla_i$ is the $G(m,p,N)$ Dunkl operator \mref{cdunkl} with $m>p$ and
$c_1=1/r$, $c_2=\frac{p}{m}(1-\frac{s m}{r})$. 
In the case $m=p$ the polynomial $p_{\nu,T}^{(m)}$ satisfies \mref{dungmp} if $c_1=1/r$ and $r = m s$.
\end{proposition}
\begin{proof}
By symmetry it is sufficient to establish that $\nabla_1 p_{\nu, T}^{(m)}=0$. Consider firstly the case when $1 \notin T$. We have
\begin{multline*}
\nabla_1 p_{\nu, T}^{(m)} =\\
= \p_1 p_{\nu, T}^{(m)}
- c_1 \sum_{j \in {\mathcal I}_1;{k=0}}^{m-1}
\frac{1-s_{1j}^k}{x_1-\xi^k x_j}
p_{\nu, T}^{(m)}-
c_1\sum_{i=2}^{r-1}
\sum_{{j \in {\mathcal I}_i}; k=0}^{m-1}
\frac{1-s_{1j}^k}{x_1-\xi^k x_j}
p_{\nu, T}^{(m)}=\\
=\p_1 p_{\nu, T}^{(m)} - 2c_1\p_1 p_{\nu,T}^{(m)} - c_1 (r-2) \p_1 p_{\nu,T}^{(m)}
\end{multline*}
by Lemma \ref{lemalt}. Since $c_1=1/r$ the value of the last expression is $0$.

Consider now the case $1 \in T$. Assume $m>p$. We have
\begin{multline*}
\nabla_1 p_{\nu, T}^{(m)}
= \p_1 p_{\nu, T}^{(m)} - \frac{mc_2}{p x_1} p_{\nu,T}^{(m)} -
c_1 \sum_{j \in {\mathcal I}_1;{k=0}}^{m-1}
\frac{1-s_{1j}^k}{x_1-\xi^k x_j}
p_{\nu, T}^{(m)}-
\\ -c_1\sum_{i=2}^{r-1}
\sum_{{j \in {\mathcal I}_i}; k=0}^{m-1}
\frac{1-s_{1j}^k}{x_1-\xi^k x_j}
p_{\nu, T}^{(m)}
=\\
=\p_1 p_{\nu, T}^{(m)} - \frac{m c_2}{p x_1} p_{\nu,T}^{(m)} - 2c_1\p_1 p_{\nu,T}^{(m)} +
\frac{2c_1}{x_1}p_{\nu,T}^{(m)} -\\
-c_1\sum_{i\in T, i \ne 1} (\p_1-\frac{1}{x_1})p_{\nu, T}^{(m)}- c_1\sum_{i\notin T} (\p_1+\frac{m-1}{x_1})p_{\nu, T}^{(m)}
\end{multline*}
by Lemma \ref{lemalt}. Therefore
\begin{multline*}
\nabla_1 p_{\nu, T}^{(m)} = \p_1 p_{\nu, T}^{(m)} (1-2c_1-c_1(r-2)) +\\
+ \frac{1}{x_1}\big(-\frac{m}{p}c_2+2c_1 + c_1(r-s-2)-c_1 s(m-1)\big) p_{\nu,T}^{(m)}=0
\end{multline*}
as required. The case $m=p$ also follows.
\end{proof}

We will need later a version of the previous
proposition for the cases $D_N$ and $B_N$. We formulate this corollary now.
Let $\nu$ be a partition of $N$ of length $l(\nu)\le k$.
Let $T$ be a subset of indexes $T\subset \{1,\ldots,k\}$. Define the polynomial
\beq{pnuT}
p_{\nu, T}(x)=p_\nu (x_1^2,\ldots, x_N^2)
\prod_{j \in T} \prod_{i=1}^{\nu_j} x_{\nu_1+\ldots+\nu_{j-1}+i},
\eeq
where $p_\nu$ is given by \mref{pnu} and for $j>l(\nu)$ we put $\nu_j=0$.
Let $K_{\nu, T}$ be the ideal generated by $\Sy_N$-images
of the polynomial $p_{\nu,T}$.

Let now $\nu=\nu_N^k$ be the partition defined in Proposition \ref{pr1}.
We define $K_{k,s}=K_{\nu_N^k, T_{k,s}}$,
where $T_{k,s}=\{k,k-1,\ldots,k-s+1\}$.
For $k=2r+1$ define $K_k=K_{k, r}$.

As a corollary from Proposition \ref{bnsingprop} we have the following.
\begin{proposition}\label{proposition:D-representation}
For $1 \le k \le N-1$, $0\le s \le k$,
the ideal $K_{k,k-s}$ is $H_c(B_N)$-invariant
if $c(e_i)=\frac12-\frac{s}{k+1}$ and $c(e_i \pm e_j)=1/(k+1)$.
For odd $k$, $1 \le k \le N-1$, the ideal $K_k$ is $H_{1/(k+1)}(D_N)$-invariant.
\end{proposition}

Below we will also need some other ideals (which we don't claim to be
representations of any interesting algebras). Namely, for $0\le s\le k$
define the ideal $\K_{k, s}$ to be generated by $\Sy_N$-images of
all polynomials \mref{pnuT} with $l(\nu)\le k$ and $|T|=s$.
Note that for $l(\nu)\le k$ and $|T|=s$ one has $K_{\nu, T}\subset \K_{k,s}$.
In particular, $K_{k,s}\subset\K_{k,s}$ and $K_{2r+1}\subset\K_{2r+1,r}$.
Note also that $\K_{k, s}\subset\K_{k+1, s}$ and $\K_{k, s}\subset\K_{k+1, s+1}$.

\begin{remark}\label{remark:my-invalidy}
The inclusions
$K_{k, s}\subset\K_{k, s}\subset I^{\pm}_k\cap J_{k-s}$ are obvious and it would be interesting to clarify if
any or both of them are actually equalities.
\end{remark}


\section{Local integrability}
\label{section:L2}

Let $X$ be a smooth variety of dimension $N\ge 2$ defined over a
field $\k$ of $\mathrm{char}(\k)=0$, and $D$ a $\Q$-divisor on $X$.
Write $D=\sum d_j D_j$, where $D_j$ are pairwise different prime
divisors. For a rational function $\Phi\in\k(X)$ we denote by
$(\Phi)$ the divisor defined by $\Phi$. By $X(\k)$ we denote the set
of $\k$-points of $X$. Recall that in the case $\k=\R$ the set
$X(\R)$ has a structure of a $C^{\infty}$-manifold provided that
$X(\R)\neq\varnothing$.

\begin{definition}[{see e.\,g.~\cite[Definition~3.3]{SingOfPairs}}]
\label{def:discrepancy}
Let $\pi:Y\to X$ be a birational morphism. Write
$$K_Y+\pi^{-1}(D)\sim_{\Q} \pi^*(K_X+D)+\sum a(E_i) E_i,$$
where $K_X$ and $K_Y$ are the canonical classes of $X$ and $Y$,
respectively, $\pi^{-1}$ and $\pi^*$ stand for a proper transform
and a pull-back, and $E_i$ are the exceptional divisors of $\pi$.
The coefficients $a(E_i)=a(X, D, E_i)\in\Q$ are called
\emph{discrepancies}.
\end{definition}

\begin{convention}[{see e.\,g.~\cite[Convention~3.3.2]{SingOfPairs}}]
\label{convention} Define the discrepancy of a
(non-ex\-cep\-ti\-o\-nal!) divisor $D_j$ to equal $a(D_j)=-d_j$.
\end{convention}

\begin{example}\label{example:discrepancy}
Let $\Phi$ be a rational function on $X$ and $c\in\Q$. Let
$\pi:Y\to X$ be a birational morphism from a smooth variety $Y$.
Take an exceptional divisor $E$ of $\pi$. Choose the local
coordinates $y_1, \ldots, y_N$ in a neighborhood of a
point $Q\in E$ so that $y_1=0$ is a local equation of $E$, and the
local coordinates $x_1, \ldots, x_N$ in a neighborhood of the point
$P=\pi(Q)$. Put
$$m=\mult_{(y_1=0)}\Phi\circ\pi \text{\ \ and\ \ }
e=\mult_{(y_1=0)}\det\Big(\frac{\partial y_i}{\partial x_j}\Big). $$
Then $a(X, c(\Phi), E)=e-cm$. (Note that this formula agrees with
Convention~\ref{convention}.)
\end{example}

\begin{example}\label{example:blow-up}
Let $X=\A^N_{\k}$, and let $\pi:Y\to X$ be a blow-up of a subvariety
$Z\subset X$ of dimension $d$ with an exceptional divisor $E$. Let
$\Phi$ be a rational function on $X$ and $c\in\Q$. Then
$$a(X, c(\Phi), E)=N-d-1-c\cdot\mult_Z(\Phi).$$
\end{example}

\begin{definition}[{see e.\,g.~\cite[1.1.3]{SingOfPairs}}]
\label{def:log-resolution} Let $\pi:Y\to X$ be a birational
morphism. We call it a \emph{log-resolution} of the pair $(X, D)$,
if $Y$ is smooth and the union of the (support of the)
strict transform $\pi^{-1}(D)$
of $D$ on $Y$ and the exceptional locus $Exc(\pi)$ is a normal
crossing divisor.
\end{definition}

\begin{definition}[{see e.\,g.~\cite[Definition~3.5]{SingOfPairs}}]
\label{definition:klt} Assume that $\k=\bar{\k}$. The pair $(X, D)$
is called \emph{Kawamata log-terminal} (or \emph{klt} for short) if
for any log-resolution $\pi:Y\to X$ 
the inequalities 
$a(F)>-1$ hold, 
where $F$ is any exceptional divisor $E_i$ of $\pi$ or any
component $D_j$ of the divisor $D$.
\end{definition}

\begin{definition}
\label{definition:klt-over-k} The pair $(X, D)$ is called Kawamata
log-terminal if such is the pair $(X_{\bar{\k}}, D_{\bar{\k}})$.
\end{definition}

\begin{remark}
\label{remark:klt-coeff} If the pair $(X, D)$ is klt, one has
$d_j<1$ for all $j$.
\end{remark}

It appears that to check the klt condition it is not necessary to
consider all possible log-resolutions.

\begin{theorem}[{see e.\,g.~Lemmas~3.10.2 and~3.12 in~\cite{SingOfPairs}}]
\label{theorem:klt-finite} Assume that $d_j<1$ for all $j$, and that
there exists a log resolution $\pi:Y\to X$ of the pair $(X, D)$ such
that the discrepancy of any exceptional divisor appearing on $Y$ is
greater than $-1$. Then the pair $(X, D)$ is klt.
\end{theorem}

Recall that in the case $\k=\R$ (or $\k=\C$) a function $\Psi$ is
said to be \emph{locally $L^1$-integrable} (or just \emph{locally
integrable}) at a point $P\in X(\k)$ if for a sufficiently small
(analytic) neighborhood $P\in U_P\subset X(\k)$ the integral
$$\int\limits_{U_P} |\Psi| dV<\infty.$$
A function $\Psi$ is said to be \emph{locally $L^2$-integrable} if
the function $\Psi^2$ is locally integrable.

One of the important applications of klt singularities is provided
by the following theorem (cf.~\cite[Corollary~2]{Saito}).

\begin{theorem}
\label{theorem:klt-integrable}%
Let $\k=\R$. Let $\Phi$ be a rational function on $X$ and $c\in\Q$.
Assume that the pair
$$\Big(X, c\big(\Phi\big)\Big)$$
is klt. Then $\Phi^{-c}$ is locally integrable on $X$.
\end{theorem}
\begin{proof}
The idea of the proof is standard (see e.\,g. the proofs
of~\cite[2.11]{WhichAreIntegrable}
or~\cite[Proposition~3.20]{SingOfPairs}), but since it is usually
given in the case $\k=\C$ (and locally $L^2$-integrable functions),
we will reproduce it here for convenience of the reader.\footnote{Another minor
difference is that we will need to work with rational rather than regular
functions, but this does not influence the proof at all.}

Let $\dim(X)=N$. We may assume that $X(\R)\neq\varnothing$.
Choose the local coordinates $x_1, \ldots, x_N$ in
a neighborhood of some point $P\in X(\R)$, and put
$dV=dx_1\wedge\ldots\wedge dx_N$. The function $\Phi^{-c}$ is
locally integrable near $P$ if and only if for some open subset
$P\in U\subset X(\R)$ the integral
$$\int_U \big|\Phi\big|^{-c}dV<
\infty.$$

Let $\pi:Y\to X$ be a log resolution of the pair
$\big(X,(\Phi)\big)$. Then
$$
\int_U \big|\Phi\big|^{-c}dV= \int_{\pi^{-1}(U)}
\big|\Phi\circ\pi\big|^{-c}\pi^*dV.
$$

Choose a point $Q\in\pi^{-1}(U)$ such that $\pi(Q)=P$, and the local
coordinates $y_1, \ldots, y_N$ in the neighborhood of $Q$. Then
$$\Phi\circ\pi=\Xi\prod y_i^{m_i} \text{\ \ and\ \ }
\pi^*dV=\Theta\prod y_i^{e_i} dy_1\wedge\ldots\wedge dy_N$$ for some
functions $\Xi$ and $\Theta$ that are invertible in a neighborhood
of $Q$, and
$$m_i=\mult_{(y_i=0)}\Phi\circ\pi,
\ \ e_i=\mult_{(y_i=0)}\det\Big(\frac{\partial y_t}{\partial
x_j}\Big).$$ Thus the initial integral is finite if and only if for
any choice of $Q$ such is the integral
$$\int\ldots\int_{U_1\times\ldots\times U_N}\prod
|y_i|^{e_i-cm_i}dy_1\wedge\ldots\wedge dy_N,$$ where $U_i\subset\R$
is some open subset. The latter holds provided that each of the
integrals
$$\int_{U_i}|y_i|^{e_i-cm_i}dy_i<\infty,$$
that is when $e_i-cm_i>-1$ for all $i$. Now the assertion follows by
Example~\ref{example:discrepancy} and Remark~\ref{remark:klt-coeff}.
\end{proof}

\begin{remark}[{cf.~\cite[Theorem~1]{Saito}}]
\label{remark:reverse}
Unlike the case $\k=\C$, the statement of
Theorem~\ref{theorem:klt-integrable} is not invertible. For example,
take $X=\R^3$ with coordinates $x_1, x_2, x_3$ and define the
divisor $D$ by the equation $\Phi=x_1^2+x_2^2+x_3^2=0$. Then for
$1\le c<3/2$ the function $\Phi^{-c}$ is integrable, but the pair
$(X, cD)$ is not klt.

On the other hand, the converse to the statement of
Theorem~\ref{theorem:klt-integrable} does hold in some important
particular cases, for example when $X=\R^N$ and the poles of
$\Phi^{-c}$ are supported on the real hyperplanes (cf.
Remark~\ref{remark:N-5} and Section~\ref{section:conclusionsnegative}).
\end{remark}

Consider now a collection of hyperplanes given by the equations
$l_i=\nlb 0$, $i=1,\,\ldots,M$, where $l_i$ are some non-zero covectors in $\R^N$.
Recall that this collection defines a semi-lattice $\LL$ which is
the minimal set of linear subspaces of $\R^N$ containing all the
hyperplanes $l_i=0$ and closed with respect to intersection.

\begin{corollary}
\label{corollary:integrable} Let $\FF\subset\R[x_1, \ldots, x_N]$ be
a finite set of polynomials, and $\bar{\FF}$ be the ideal generated
by $\FF$. Choose the numbers $c_i\in\Q$, $i=1,\,\ldots,M$. For a linear subspace
$L\subset\R^N$ define $m(L)=m_{\FF}(L)$ to be the minimal multiplicity of a
function $f\in\FF$ along $L$, and $\kappa(L)=\sum_{L\subset
l_i}c_i$. Then for any $f\in\bar{\FF}$ the function
$$f\prod_{i=1}^M l_i^{-c_i}$$
is locally $L^2$-integrable at any point $P\in\R^N$ provided that
$$
\kappa(L)<\frac{\codim(L)}{2}+m(L)
$$
for any $L\in\LL$.
\end{corollary}
\begin{proof}
Choose a nonzero function $f\in\bar{\FF}$ and put
$$\Phi_f=\frac{\prod l_i^{c_i}}{f}.$$
By Theorems~\ref{theorem:klt-finite}, \ref{theorem:klt-integrable} it is enough to check that
the discrepancies
$$a(\R^N, (\Phi_f^2), E_j)>-1$$
for all exceptional
divisors $E_i$ of a partial log-resolution $\pi:Y\to\R^N$ such that
$\pi^{-1}(\bigcup l_i)\cup Exc(\pi)$ is a normal crossing divisor.
To construct such resolution put
$\pi=\pi_{N-2}\circ\ldots\circ\pi_{0}$, where $\pi_0:Y_0\to\R^N$ is
the blow-up of the point $\text{\textbf{0}}\in\nlb\R^N$, and
$\pi_d:Y_d\to Y_{d-1}$ for $d\ge 1$ is the blow-up of the strict
transforms of all subspaces $L\in\LL$ such that $\dim(L)=d$. Note
that these strict transforms are disjoint on $Y_{d-1}$, so that
$\pi:Y=Y_{N-2}\to\R^N$ indeed enjoys the desired property. Note that
$\pi_{d-1}:Y_{d-1}\to\R^N$ is an isomorphism at a neighborhood of a
general point $P\in\pi_{d-1}^{-1}(L)$ for $L\in\LL$ with
$\dim(L)=d$. Hence the discrepancy $a(\R^N, (\Phi_f^2), E_L)$ of the
exceptional divisor $E_L$ whose center on $\R^N$ is $L$ equals the
discrepancy of the exceptional divisor of the blow-up of $\R^N$
along $L$, which in turn equals
\begin{multline*}
a_L=\codim(L)-1-2\sum c_i\mult_{L}(l_i)+2\mult_{L}(f)=\\
=\codim(L)-1-2\kappa(L)+2\mult_{L}(f)
\end{multline*}
by Example~\ref{example:blow-up}. Hence for $\Phi_f^{-1}$ to be
locally $L^2$-integrable at any $P\in\R^N$ it is enough to satisfy
the inequality
$$
\kappa(L)<\frac{\codim(L)}{2}+\mult_L(f)
$$
for any $L\in\LL$. The required assertion follows since
$$\mult_L(f)\ge\min_{\phi\in\FF}\mult_L(\phi)=m(L).$$
\end{proof}

In the case of singularities of constant order the previous corollary can be rephrased as follows.
\begin{corollary}
\label{corollary:integrable-equal-c} In the above notations assume
that $l_i$ is not proportional to $l_j$ for $i\ne j$ and that $c_i=c$ for all~$1\le i\le M$. Then
for any $f\in\bar{\FF}$ the function
$$f\prod_{i=1}^M l_i^{-c}$$
is locally $L^2$-integrable at any point $P\in\R^N$ provided that
$$
c<\min\limits_{L\in\LL}\frac{\frac{1}{2}\codim(L)+m(L)}{K(L)},
$$
where $K(L)$ is the number of $l_i$ vanishing along $L$.
\end{corollary}


\section{Elementary estimates}
\label{section:brute-force}

In this section we collect a few technical lemmas which we will need
in Section \ref{section:conclusions}.

Consider the $m$-dimensional space $\R^m\supset\Z^m$.
For any $\xi\in \R_{\ge 0}$ define
$$\Sympl_{m, \xi}=
\{(t_1,\ldots, t_m)\in\R^m \mid 0\le t_i\le \xi, \sum t_i=\xi\}\cap\Z^m.$$
Having fixed $k\in\Z_{+}$, for any $q\in\R$ we define $0\le\rho_k(q)<k$  to
satisfy $\rho_k(q)=q \ \mathrm{mod}\ k$. For $\alpha\in\R$ we denote
the integer part of $\alpha$ by $\lfloor\alpha\rfloor$.

\begin{lemma}
\label{lemma:quadratic-minimum}
Fix $\Lambda\in\Z_{+}$ and put $\rho=\rho_{m}(\Lambda)$.
Consider the function
$$C(t)=C_{m, \Lambda}(t_1,\ldots,t_m)=\sum\limits_{i=1}^m\frac{t_i(t_i-1)}{2}.
$$
Put
$$\mu_{m, \Lambda}=\min\limits_{t\in \Sympl_{m, \Lambda}}C(t).$$
Then
$$\mu_{m,\Lambda}=
\frac{\Lambda(\Lambda-m)}{2m}+\frac{\rho(m-\rho)}{2m}.$$
\end{lemma}
\begin{proof}
Note that for any $u,v\in\R$ one has
$$(u-1)^2+(v+1)^2<u^2+v^2$$
provided that $u>v+1$. Thus the minimum
$$
\mu_{m, \Lambda}=\frac{1}{2}\min\limits_{t\in\Sympl_{m,\Lambda}}
\Big(\sum\limits_{i=1}^m t_i^2\Big)-\frac{\Lambda}{2}
$$
is attained at a point $A=(a_1,\ldots,a_m)\in\Sympl_{m,\Lambda}$ such that
for any $i$ and $j$ one has $|a_i-a_j|\le 1$.
Since $C(t)$ is invariant under permutations of coordinates, we may assume
that
$$
A=\Big(\underbrace{\blfloor\frac{\Lambda}{m}\brfloor+1, \ldots,
\blfloor\frac{\Lambda}{m}\brfloor+1}_{\rho},
\underbrace{\blfloor\frac{\Lambda}{m}\brfloor, \ldots,
\blfloor\frac{\Lambda}{m}\brfloor}_{m-\rho}\Big).
$$
Write $\Lambda=sm+\rho$ for some $s\in\Z_{\ge 0}$.
One has
\begin{multline*}
\mu_{m,\Lambda}=C(A)=\frac{1}{2}
\Big(\rho\Big(\blfloor\frac{\Lambda}{m}\brfloor+1\Big)
\blfloor\frac{\Lambda}{m}\brfloor+
\big(m-\rho)\Big(\blfloor\frac{\Lambda}{m}\brfloor-1\Big)
\blfloor\frac{\Lambda}{m}\brfloor=\\
=\frac{1}{2}\big(\rho(s+1)s+(m-\rho)(s-1)s\big)=
\frac{m^2(s^2-s)+2\rho ms}{2m}=\\
=\frac{(ms+\rho)(ms+\rho-m)+\rho(m-\rho)}{2m}=\\
=\frac{\Lambda(\Lambda-m)}{2m}+\frac{\rho(m-\rho)}{2m}.
\end{multline*}
\end{proof}


\begin{lemma}\label{lemma:bolshie-ostatki-ab}
Take $a, b, z\in\Z_{+}$. Choose
$\alpha\in\R$ such that $0\le\alpha\le ab/(a+b)^2$ and put
$$\Lambda_1=\frac{za}{a+b}-\alpha (a+b), \quad
\Lambda_2=\frac{zb}{a+b}+\alpha (a+b).$$
Let
$$F=b\rho_a(\Lambda_1)\big(a-\rho_a(\Lambda_1)\big)+
a\rho_b(\Lambda_2)\big(b-\rho_b(\Lambda_2)\big).$$
Then
$$
F\ge \alpha(a+b)\big(ab-\alpha(a+b)^2\big).$$
\end{lemma}
\begin{proof}
Write $\Lambda_1=sa+\rho_a(\Lambda_1)$, $s\in\Z_{\ge 0}$. Then
$$\Lambda_2=sb+\rho_a(\Lambda_1)\frac{b}{a}+\alpha\frac{(a+b)^2}{a},$$
and
$$0\le\rho_a(\Lambda_1)\frac{b}{a}+\alpha\frac{(a+b)^2}{a}<2b$$
by assumptions.
Put
$$\mathcal{S}_1=\Big[0, \frac{a}{b}\big(b-\alpha\frac{(a+b)^2}{a}\big)\Big)
\subset [0, a)\subset\R,$$
and
$$\mathcal{S}_2=[0, a)\setminus\mathcal{S}_1=
\Big[\frac{a}{b}\big(b-\alpha\frac{(a+b)^2}{a}\big), a\Big)\subset\R.$$

Suppose first that $\rho_a(\Lambda_1)\in\mathcal{S}_1$. Then
$$\rho_a(\Lambda_1)\frac{b}{a}+\alpha\frac{(a+b)^2}{a}<b,$$
and
$$\rho_b(\Lambda_2)=\rho_a(\Lambda_1)\frac{b}{a}+\alpha\frac{(a+b)^2}{a}.$$
Note that $F=F\big(\rho_a(\Lambda_1)\big)$
is a quadratic function in $\rho_a(\Lambda_1)$ with
negative coefficient at $\rho_a(\Lambda_1)^2$.
Thus
\begin{multline*}
\inf\limits_{\rho_a(\Lambda_1)\in\mathcal{S}_1}F\big(\rho_a(\Lambda_1)\big)
\ge \min\Big\{F\big(0\big),
F\Big(\frac{a}{b}\big(b-\alpha\frac{(a+b)^2}{a}\big)\Big)\Big\}=\\
=
\min\Big\{\alpha\frac{(a+b)^2}{a}\big(ab-\alpha(a+b)^2\big),
\alpha\frac{(a+b)^2}{b}\big(ab-\alpha(a+b)^2\big)\Big\}\ge \\
\ge \alpha(a+b)(ab-\alpha(a+b)^2).
\end{multline*}

Suppose now that $\rho_a(\Lambda_1)\in\mathcal{S}_2$. Then
$$b\le \rho_a(\Lambda_1)\frac{b}{a}+\alpha\frac{(a+b)^2}{a}<2b,$$
and
$$\rho_b(\Lambda_2)=\rho_a(\Lambda_1)\frac{b}{a}+\alpha\frac{(a+b)^2}{a}-b.$$
Note that $F=F\big(\rho_a(\Lambda_1)\big)$
is again a quadratic function in $\rho_a(\Lambda_1)$ with
negative coefficient at $\rho_a(\Lambda_1)^2$.
Thus
\begin{multline*}
\inf\limits_{\rho_a(\Lambda_1)\in\mathcal{S}_2}F\big(\rho_a(\Lambda_1)\big)
\ge\min\Big\{F\Big(\frac{a}{b}\big(b-\alpha\frac{(a+b)^2}{a}\big)\Big),
F\big(a\big)\Big\}=\\
=\min\Big\{\alpha\frac{(a+b)^2}{a}\big(ab-\alpha(a+b)^2\big),
\alpha\frac{(a+b)^2}{b}\big(ab-\alpha(a+b)^2\big)\Big\}\ge \\
\ge \alpha(a+b)(ab-\alpha(a+b)^2).
\end{multline*}
\end{proof}

\begin{lemma}\label{lemma:krivoj-krokodil-ab}
Fix $a, b, z\in\Z_{+}$.
Consider the function
$$\tilde{C}(t)=\tilde{C}(t_1,\ldots,t_{a+b})=
\sum\limits_{i=1}^a t_i+
\sum\limits_{i=1}^{a+b}t_i(t_i-1).$$
Put
$$\tilde \mu = \tilde{\mu}_{a, b, z}=
\min\limits_{t\in \Sympl_{a+b,z}}\tilde{C}(t).$$
Then
\beq{lemineq}
\tilde{\mu}\ge \frac{z^2}{a+b}-\frac{b}{a+b}z.
\eeq
\end{lemma}
\begin{proof}
Denote the $a$-tuple $(t_1, \ldots, t_{a})$ by $T_1$, and denote the
$b$-tuple $(t_{a+1}, \ldots, t_{a+b})$ by $T_2$. Put $\sum_{i=1}^a
t_i=\Lambda_1$ and $\Lambda_2=z-\Lambda_1$. One has
\begin{multline*}
\tilde{\mu}=\min\limits_{\Lambda_1+\Lambda_2=z}\Big(
\min\limits_{t\in\Sympl_{a+b, z}}\big(\Lambda_1+
\sum\limits_{i=1}^{a}t_i(t_i-1)
+\sum\limits_{i=a+1}^{a+b}t_i(t_i-1)
\big)\Big)
=\\
=\min\limits_{\Lambda_1+\Lambda_2=z}\big(
\Lambda_1+2\min\limits_{T_1\in\Sympl_{a,\Lambda_1}}
C_{a}(T_1)+
2\min\limits_{T_2\in\Sympl_{b,\Lambda_2}}
C_{b}(T_2)\big)=\\
=\min\limits_{\Lambda_1+\Lambda_2=z}\big(\Lambda_1+
2\mu_{a, \Lambda_1}+2\mu_{b, \Lambda_2}\big)=\\
=\min\limits_{\Lambda_1+\Lambda_2=z}\Big(\Lambda_1+
\frac{\Lambda_1(\Lambda_1-a)+\rho_a(\Lambda_1)(a-\rho_a(\Lambda_1))}{a}
+\\
+\frac{\Lambda_2(\Lambda_2-b)+\rho_b(\Lambda_2)(b-\rho_b(\Lambda_2))}{b}
\Big)=\\
=\min\limits_{\Lambda_1+\Lambda_2=z}
\Big(\frac{b\Lambda_1^2+a\Lambda_2^2-ab\Lambda_2}{ab}+\\
+\frac{b\rho_a(\Lambda_1)(a-\rho_a(\Lambda_1))
+a\rho_b(\Lambda_2)(b-\rho_b(\Lambda_2))}{ab}\Big).
\end{multline*}

Suppose that the minimum is attained for $\Lambda_1=az/(a+b)-\alpha (a+b)$,
$\Lambda_2=zb/(a+b)+\alpha (a+b)$ (note that we don't assume that $\alpha$
is nonnegative). Then
\begin{multline*}
\frac{b\Lambda_1^2+a\Lambda_2^2-ab\Lambda_2}{ab}=
\frac{z^2}{a+b}-\frac{b}{a+b}z+
\alpha\frac{a+b}{ab}\big(\alpha(a+b)^2-ab\big)
\end{multline*}
Thus to conclude the proof we may assume that
$\alpha(\alpha(a+b)^2-ab)\le 0$, i.\,e. $0\le\alpha\le ab/(a+b)^2$,
and the assertion follows by Lemma~\ref{lemma:bolshie-ostatki-ab}.
\end{proof}

Choose nonnegative integers $N\ge 2$ and $z\le N$, and let
$\lambda=(\lambda_1, \ldots, \lambda_{l(\lambda)})$
be a partition of $N-z$, i.\,e.
$$\lambda_1\ge\ldots\ge\lambda_{l(\lambda)}>0 \text{ and }
\sum\lambda_i=N-z.$$
(In particular, we allow an ``empty'' partition when $l(\lambda)=0$
and $z=\nlb N$.)
Put
\beq{Rklambda}
R_k(\lambda)=\sum_{i=1}^{l(\lambda)}
\rho_k(\lambda_i)(k-\rho_k(\lambda_i)).
\eeq

\begin{lemma}\label{lemma:type-B-crocodile}
Let $k\in\Z_{+}$, and $\lambda$ be a partition of $N-z$ as above.
Then
\begin{multline}
\label{eq:B-croco}
\sum_{i=1}^{l(\lambda)}\frac{\lambda_i(\lambda_i-k)}{k}+
\frac{R_k(\lambda)}{k}+ N-l(\lambda)\ge
\sum_{i=1}^{l(\lambda)}\frac{\lambda_i(\lambda_i-1)}{k}+z.
\end{multline}
\end{lemma}
\begin{proof}
Denote by $v$ the number of $\lambda_i$ which are divisible by $k$.
Then $N-z \ge kv + l(\lambda)-v$, and $R_k(\lambda) \ge (k-1)(l(\lambda)-v)$,
hence
$$
N+R_k(\lambda) \ge kl(\lambda) +z.
$$

One has
\begin{multline*}
\sum_{i=1}^{l(\lambda)}\frac{\lambda_i(\lambda_i-k)}{k}+
\frac{R_k(\lambda)}{k}+ N-l(\lambda)=
\sum_{i=1}^{l(\lambda)}\frac{\lambda_i^2}{k}+
\frac{R_k(\lambda)}{k} + z - l(\lambda)=\\
=
\sum_{i=1}^{l(\lambda)}\frac{\lambda_i^2}{k}+
\frac{N+R_k(\lambda)-kl(\lambda)-z}{k}-\frac{N-z}{k}
+ z \ge \\
\ge
\sum_{i=1}^{l(\lambda)}\frac{\lambda_i^2}{k}-
\frac{\sum_{i=1}^{l(\lambda)}\lambda_i}{k}+
z=
\sum_{i=1}^{l(\lambda)}\frac{\lambda_i(\lambda_i-1)}{k}+
z.
\end{multline*}
\end{proof}

\begin{lemma}\label{lemma:type-A-crocodile}
Let $k\in\Z_{+}$,
and let $\lambda$ be a partition of $N$.
Assume that $\lambda_1>1$.
Then
$$
\sum_{i=1}^{l(\lambda)}\frac{\lambda_i(\lambda_i-k)}{k}+
\frac{R_k(\lambda)}{k}+N-l(\lambda)
>
\sum_{i=1}^{l(\lambda)}\frac{\lambda_i(\lambda_i-1)}{k+1}.
$$
\end{lemma}
\begin{proof}
By Lemma~\ref{lemma:type-B-crocodile} applied for $z=0$ one has
\begin{multline*}
\sum_{i=1}^{l(\lambda)}\frac{\lambda_i(\lambda_i-k)}{k}+
\frac{R_k(\lambda)}{k}+N-l(\lambda)\ge
\\
\ge
\sum_{i=1}^{l(\lambda)}\frac{\lambda_i(\lambda_i-1)}{k}
\ge
\sum_{i=1}^{l(\lambda)}\frac{\lambda_i(\lambda_i-1)}{k+1}.
\end{multline*}
Moreover, the last inequality is strict if the right hand side is non-zero,
which happens exactly when $\lambda_1>1$.
\end{proof}

\begin{lemma}\label{lemma:type-B-crocodile-old}
Let $\lambda$ be a partition of $N-z$. Assume that
$z \ge 1$
if $\lambda_1=1$. Then
\begin{multline}
\label{lem46}
\sum_{i=1}^{l(\lambda)}\frac{\lambda_i(\lambda_i-k)}{2k}+
\frac{R_k(\lambda)}{2k}+ \frac{z(z-k)}{k}+\frac{\rho_k(z)(k-\rho_k(z))}{k}+\\
+\max(0,\lfloor z-s\rfloor)+\frac{N-l(\lambda)}{2} >
\\
>
\sum_{i=1}^{l(\lambda)}\frac{\lambda_i(\lambda_i-1)}{2(k+1)}+
\frac{z(z-1)}{k+1}+
\frac{\frac{k+1}2-s}{k+1}z,
\end{multline}
where $k \in \Z_+$, $s \in \R$.
\end{lemma}
\begin{proof}
It is sufficient to establish that the left-hand side of \mref{lem46}
is greater or
equal than
\beq{rhs}
\frac{1}{2k} \sum \lambda_i^2 -\frac{N-z}{2k}
+\frac{z^2}{k+1}+ \frac{\frac{k-1}{2}-s}{k+1}z,
\eeq
and that it is
strictly bigger than \mref{rhs} when $\lambda_1=1$ or $l(\lambda)=0$
(which are exactly the cases when the first of the three summands in the
right-hand side of \mref{lem46} vanishes).

By Lemma \ref{lemma:type-B-crocodile} our statement is implied by the
inequality \beq{lemmaBmainid} \frac{z(z-k)}{k}
+\frac{\rho_k(z)(k-\rho_k(z))}{k}+\max(0, \lfloor z-s\rfloor ) >
\frac{z^2}{k+1} - \frac{s+1}{k+1}z. \eeq Moreover, we may assume
that $z>0$. Indeed, the case $z=0$ leads to the equality in
\mref{lemmaBmainid}, but in this case $l(\lambda)>0$ and
$\lambda_1>1$ by assumption.

It is clear that the inequality \mref{lemmaBmainid}
holds for $s \ge k$ as in this case
$$
\frac{z(z-k)}{k}
> \frac{z^2}{k+1} -
\frac{s+1}{k+1}z
$$
so let us suppose that $s<k$. Now consider few possible cases for
the values of $z$.  When $z <s+1$ the left-hand side of
\mref{lemmaBmainid} equals 0 and the inequality holds. When $s+1 \le
z \le k$ the left-hand side of \mref{lemmaBmainid} equals $\lfloor
z-s\rfloor$ so the inequality \mref{lemmaBmainid} takes the form
$$
z^2 - (s+1)z-(k+1)\lfloor z-s\rfloor<0.
$$
This inequality is correct since $-\lfloor z-s\rfloor < 1-z+s$ and
$$z^2-(s+1)z+(k+1)(1-z+s)=(z-s-1)(z-k-1)\le 0.$$
Finally, when $z\ge k+1$ the left-hand side of \mref{lemmaBmainid}
is bigger than
$$
\frac{z^2}{k}-s-1,
$$
and \mref{lemmaBmainid} holds.
\end{proof}

\begin{remark}\label{remarkafterlem}
It follows from the proof of Lemma~\ref{lemma:type-B-crocodile-old} that
its assertion remains true when $\max(0, \lfloor z-s\rfloor)$ is
replaced by $0$ in the case $s \ge k$.
\end{remark}


\section{Conclusions on unitarity for classical root systems}
\label{section:conclusions}

In this section we apply the previous estimates
to establish unitarity of certain submodules
in the polynomial representation. We start with the $A_{N-1}$ case.

\begin{proposition}
\label{proposition:A-type}
In the notations of Section~\ref{section:ideals} the function
$$\frac{f}{\prod_{i<j}^N (x_i-x_j)^c}$$
is locally $L^2$-integrable for all $f\in I_k$ \lb$1\le k \le
N-1$\rb\ provided that $c\le 1/(k+1)$.
\end{proposition}
\begin{proof}
Assume the notation of Corollary~\ref{corollary:integrable}, and
consider the semi-lattice $\LL$ generated by the hyperplanes
$l_{ij}=x_i-x_j=0$. By Corollary~\ref{corollary:integrable-equal-c}
it is enough to check that
$$\frac{1}{k+1} < \min\limits_{L\in\LL}\frac{\frac12\codim(L)+m(L)}{K(L)},$$
where $m(L)=m_{I_k}(L)$.
By $\Sy_N$-symmetry it suffices to consider the linear
subspaces $L=L_{\lambda}$ given by
\begin{multline*}
x_1=\ldots=x_{\lambda_1}, x_{\lambda_1+1}=\ldots=x_{\lambda_1+\lambda_2},
\ldots,\\
x_{\lambda_1+\ldots+\lambda_{l(\lambda)-1}+1}=\ldots=
x_{\lambda_1+\ldots+\lambda_{l(\lambda)}}
\end{multline*}
for some partition $\lambda=(\lambda_1, \ldots, \lambda_{l(\lambda)})$ of $N$ where $\lambda_1 >1$.
It is easy to see that $\codim(L)=N-l(\lambda)$, and
$$K(L)=\sum\frac{\lambda_i(\lambda_i-1)}{2}.$$
To compute $m(L)$ consider  $\nu=(\nu_1, \ldots, \nu_{k})$ a partition of $N$ and
the corresponding polynomial $p_{\nu}\in I_k$ introduced
in Section~\ref{section:ideals}. A polynomial $\bar{p}_{\nu}$
from the $\Sy_N$-orbit of $p_{\nu}$
gives rise to a presentation of each $\lambda_i$ as
a sum of $k$ nonnegative summands
$$\lambda_i=\lambda_{i, 1}+\ldots+\lambda_{i, k},$$
so that
$$\nu_j=\lambda_{1, j}+\ldots+\lambda_{k, j}.$$
Moreover,
$$\mult_L(\bar{p}_{\nu})=\sum\limits_{i=1}^{l(\lambda)}\sum\limits_{j=1}^{k}
\frac{\lambda_{ij}(\lambda_{ij}-1)}{2}.$$
Recall that by
Proposition~\ref{pr1} the $\Sy_N$-orbits of the polynomials $p_{\nu}$
for various $\nu$ generate the ideal $I_k$.
Hence, in the notation of Lemma~\ref{lemma:quadratic-minimum} one has
$$
m(L)=\min\limits_{\sum_j\lambda_{ij}=\lambda_i}\Big(
\sum\limits_{i=1}^{l(\lambda)}\sum\limits_{j=1}^{k}
\frac{\lambda_{ij}(\lambda_{ij}-1)}{2}\Big)
=\sum\limits_{i=1}^{l(\lambda)}\mu_{k, \lambda_i}.
$$
By Lemma~\ref{lemma:quadratic-minimum}
one has
$$m(L)=\sum\limits_{i=1}^{l(\lambda)}\frac{\lambda_i(\lambda_i-k)}{2k}+
\frac{R_k(\lambda)}{2k},$$
where $R_k(\lambda)$ is defined by $\mref{Rklambda}$. The desired assertion is implied by
Lemma~\ref{lemma:type-A-crocodile}.
\end{proof}

The ideal $I_k$ is an (irreducible) representation of the rational
Cherednik algebra $H_{c}(\Sy_N)$ when $c=1/(k+1)$. Therefore
Proposition~\ref{proposition:A-type} has the following corollary
which was firstly established in \cite{ESG} by different arguments.

\begin{corollary}[{\cite[Theorem 5.14]{ESG}}]
The representation $I_k$ is a unitary representation of the rational
Cherednik algebra $H_{1/(k+1)}(\Sy_N)$.
\end{corollary}


Now we move to the $D_N$ and $B_N$ cases. We are going to establish
local $L^2$-integrability of the relevant functions based on the
polynomials from the ideal $I_k^\pm$. In order to do this we
consider the subspaces from the intersection semi-lattice
$\mathcal{L}$ of the arrangement of hyperplanes of type $D_N$.
Namely, we say that a linear space is {\it of type $(\lambda,z)$}
where $0 \le z \le N$ and $\lambda=(\lambda_1, \ldots,
\lambda_{l(\lambda)})$ is a partition of $N-z$ if the space is a
$D_N$-image of the following linear space:
\begin{multline}\label{lambdaz}
x_1=\ldots=x_{\lambda_1}, x_{\lambda_1+1}=\ldots=x_{\lambda_1+\lambda_2},
\ldots,\\
x_{\lambda_1+\ldots+\lambda_{l(\lambda)-1}+1}=\ldots=
x_{\lambda_1+\ldots+\lambda_{l(\lambda)}},
x_{\lambda_1+\ldots+\lambda_{l(\lambda)} +1 }= \ldots
= x_N =0.
\end{multline}
For a fixed subspace $L$ we will refer to the variables
involved in the last group of equations as \emph{$z$-variables}, and to the
other variables as \emph{$\lambda$-variables}.

Note that any element $L \in \mathcal{L}$ has above
type with $z\ne 1$ except the case when $N$ is even and $z=0$. In
this case $\mathcal{L}$ also contains the spaces {\it of type
$\lambda^-$} given by the $D_N$-images of the linear space
determined by the equations
\begin{multline*}
-x_1=x_2=\ldots=x_{\lambda_1}, x_{\lambda_1+1}=\ldots=x_{\lambda_1+\lambda_2},
\ldots,\\
x_{\lambda_1+\ldots+\lambda_{l(\lambda)-1}+1}=\ldots=x_N,
\end{multline*}
where $\lambda$ is a partition of $N$.

Recall that ideals $K_k$ and $\K_{k, s}$ were defined in the end
of Section~\ref{section:ideals}, and
by Proposition~\ref{proposition:D-representation}
the ideal $K_{2r-1}$ is a representation of the algebra $H_c(D_N)$ for $c=1/(2r)$.

\begin{theorem}
\label{proposition:D-type} The function
$$
\frac{f}{\prod_{i<j}^N (x_i^2-x_j^2)^{\frac{1}{2r}}}
$$
is locally
$L^2$-integrable  for all $f\in K_{2r-1}$ provided that
$r\le \frac{N}2$, $r \in \Z_+$.
\end{theorem}
\begin{proof}
Assume the notations of Corollary~\ref{corollary:integrable}, and
consider the semi-lattice $\LL$ generated by the hyperplanes
$l_{ij}=x_i-x_j=0$ and $l_{ij}'=x_i+x_j=0$.
By Corollary~\ref{corollary:integrable-equal-c}
it is enough to check that
\beq{corD2r}\frac{1}{2r}<\min\limits_{L\in\LL}\frac{\frac12 \codim(L)+m(L)}{K(L)},
\eeq
where $m(L)=m_{K_{2r-1}}(L)$

Choose a subspace $L\in\LL$ of type $(\lambda, z)$
where $0\le z\le N$, $z\neq 1$, and
$\lambda=(\lambda_1, \ldots, \lambda_{l(\lambda)})$ is a partition of $N-z$.
It is easy to see that $\codim(L)=N-l(\lambda)$ and
$$K(L)=\sum\frac{\lambda_i(\lambda_i-1)}{2}+z(z-1).$$

Since $K_{2r-1}\subset \K_{2r-1, r-1}\subset\K_{2r, r}$,
one has $$m(L)\ge m_{\K_{2r-1, r-1}}(L)\ge m_{\K_{2r,r}}(L).$$

Assume first that $z>0$. Let us estimate the value of
$m_{\K_{2r, r}}(L)$. Consider a partition
$\nu=(\nu_1, \ldots, \nu_{2r})$, a set
$$T=\{\tau_1, \ldots, \tau_r\}\subset\{1, \ldots, 2r\}$$
and
the corresponding polynomial $p_{\nu, T}\in\K_{2r, r}$ introduced
in Section~\ref{section:ideals}. A polynomial $\bar{p}_{\nu, T}$
from the $\Sy_N$-orbit of $p_{\nu, T}$
gives rise to a presentation of each $\lambda_i$ and $z$ as
a sum of $2r$ nonnegative summands
$$\lambda_i=\lambda_{i, 1}+\ldots+\lambda_{i, 2r},\quad
z=\zeta_1+\ldots+\zeta_{2r}$$
so that
$$\nu_j=\lambda_{1, j}+\ldots+\lambda_{l(\lambda), j}+\zeta_j.$$
Moreover,
$$\mult_L(\bar{p}_{\nu})=\sum\limits_{i=1}^{l(\lambda)}\sum\limits_{j=1}^{2r}
\frac{\lambda_{ij}(\lambda_{ij}-1)}{2}
+\sum\limits_{j=1}^{2r}\zeta_j(\zeta_j-1)+
\sum\limits_{\tau\in T}\zeta_{\tau}.$$
Hence, in the notation of Lemma~\ref{lemma:quadratic-minimum}
and Lemma~\ref{lemma:krivoj-krokodil-ab} one has
\begin{multline*}
m_{\K_{2r, r}}(L)=\\
=\sum\limits_{i=1}^{l(\lambda)}
\Big(\min\limits_{\sum_j\lambda_{ij}=\lambda_i}
\sum\limits_{j=1}^{2r}
\frac{\lambda_{ij}(\lambda_{ij}-1)}{2}\Big)+
\min\limits_{\sum_j\zeta_j=z}\Big(\sum_{j=1}^{2r}\zeta_j(\zeta_j-1)+
\sum\limits_{\tau\in T}\zeta_{\tau}\Big)=\\
=\sum\limits_{i=1}^{l(\lambda)}\mu_{2r, \lambda_i}+\tilde{\mu}_{r, r, z}.
\end{multline*}
By Lemmas~\ref{lemma:quadratic-minimum} and~\ref{lemma:krivoj-krokodil-ab}
applied for $a=b=r$
one has
$$m(L)\ge
\sum_{i=1}^{l(\lambda)}\frac{\lambda_i(\lambda_i-2r)}{4r}+
\frac{R_{2r}(\lambda)}{4r}+ \frac{z^2}{2r}-\frac{z}{2}.$$
By
Lemma~\ref{lemma:type-B-crocodile} applied for $k=2r$ 
one has
\begin{multline*}
\sum_{i=1}^{l(\lambda)}\frac{\lambda_i(\lambda_i-2r)}{4r}+
\frac{R_{2r}(\lambda)}{4r}+ \frac{z^2}{2r}-\frac{z}{2}
+\frac{N-l(\lambda)}{2}\ge\\
\ge \sum_{i=1}^{l(\lambda)}\frac{\lambda_i(\lambda_i-1)}{4r}+\frac{z^2}{2r}>\\
>\sum_{i=1}^{l(\lambda)}\frac{\lambda_i(\lambda_i-1)}{4r}+\frac{z(z-1)}{2r} = \frac{K(L)}{2r}
\end{multline*}
so \mref{corD2r} follows.

Now assume that $z=0$ and estimate $m_{\K_{2r-1,r-1}}$.
Arguing as in the proof of Proposition~\ref{proposition:A-type}, one
obtaines
\begin{multline*}
m(L)\ge m_{\K_{2r-1,r-1}}(L)=m_{I^{\pm}_{2r-1}}(L)=m_{I_{2r-1}}(L)=
\sum_{i=1}^{l(\lambda)}\mu_{2r-1,\lambda_i}=
\\
=\sum_{i=1}^{l(\lambda)}\frac{\lambda_i(\lambda_i-(2r-1))}{2(2r-1)}
+\frac{R_{2r-1}(\lambda)}{2(2r-1)}.
\end{multline*}
Thus the assertion in this case is implied by
Lemma~\ref{lemma:type-A-crocodile} applied for $k=2r-1$.

Finally, choose a subspace $L\in\LL$ of type $\lambda^{-}$. It is easy to
see that the values of $\codim(L)$, $K(L)$ and $m(L)$ are the same
as for a subspace of type $(\lambda, 0)$, which completes the proof.
\end{proof}

Now we consider singular values $c=1/(2r)$ with $r>N/2$.
We need to use ideals $J_r$ from Section \ref{section:ideals}.

\begin{theorem}
\label{proposition:D-type-second} The function
$$
\frac{f}{\prod_{i<j}^N (x_i^2-x_j^2)^{\frac{1}{2r}}}
$$
is locally
$L^2$-integrable  for all $f\in J_{r}$ provided that
$N> r > \frac{N}2$, $r \in \Z_+$.
\end{theorem}
\begin{proof}
Let $L$ be a subspace of type $(\lambda,z)$ or $\lambda^{-}$. Note
that the multiplicity $m(L)=\max(0, z-r)$. We need to establish that
\begin{equation}\label{eq:D-J}
\frac{z(z-1)}{2r}+\sum_{i=1}^{l(\lambda)}
\frac{\lambda_i(\lambda_i-1)}{4r} < \frac{N-l(\lambda)}{2}+
\max(0,z-r).
\end{equation}

Assume that $z>0$. Then
\begin{equation}\label{eq:smalllemma}
\frac{z(z-1)}{2r}<\frac{z}{2}+\max(0,z-r),
\end{equation}
which can
be easily seen by considering the cases $2r>z\ge r$ and $z<r$.
Moreover, applying Lemma~\ref{lemma:type-B-crocodile} with $k=2r$ one obtains
\begin{equation}\label{eq:B-croco-for-D}
\sum\limits_{i=1}^{l(\lambda)}\frac{\lambda_i(\lambda_i-1)}{4r}+\frac{z}{2}\le
\frac{N-l(\lambda)}{2} \end{equation} since the first two summands
of the left hand side of~\mref{eq:B-croco} make $0$ for $k>N$.
Adding up~\mref{eq:smalllemma} and~\mref{eq:B-croco-for-D} one
obtains~\mref{eq:D-J}.

Now assume that $z=0$. Then~\mref{eq:D-J} becomes
$$\sum_{i=1}^{l(\lambda)}
\frac{\lambda_i(\lambda_i-1)}{4r}<\frac{N-l(\lambda)}{2}.$$ Since
$2r>N$ it is enough to check that
\begin{equation}\label{eq:max-lambda}
\sum_{i=1}^{l(\lambda)}\lambda_i^2 \le N^2-Nl(\lambda)+N. \end{equation} The
maximum of the left hand side of~\mref{eq:max-lambda} is obtained
for
$$\lambda=(N-l(\lambda)+1, 1,\ldots, 1).$$
Thus~\mref{eq:max-lambda} holds since
$$N^2-Nl(\lambda)+N -
(N-l(\lambda)+1)^2-(l(\lambda)-1)=(N-l(\lambda))(l(\lambda)-1)\ge 0.$$
\end{proof}

Now we move to the case
of the poles supported on the $B_N$ semi-lattice. Consider the ideal
$K_{r-1,r-s-1}$ as a representation of the rational Cherednik algebra
$H_c(B_N)$ where the multiplicity function $c(e_i \pm e_j)=1/r$
and $c(e_i)=\frac12-\frac{s}{r}$
(see Proposition~\ref{proposition:D-representation}).
Any element from the corresponding intersection
semi-lattice ${\mathcal{L}}(B_N)$ is the image of the space of the
form \mref{lambdaz} under an element of the group $B_N$. We say that
these spaces have type $(\lambda,z)$ where $z=0,1, \ldots, N$ and
$\lambda$ is a partition of $N-z$.

\begin{theorem}
\label{proposition:B-type}
The function
$$g=\frac{f}{\prod_{i<j}^N (x_i^2-x_j^2)^{\frac{1}{r}} \prod_{i=1}^N x_i^{\frac12-\frac{s}{r}}}$$
is locally $L^2$-integrable for any $f\in K_{r-1, r-s-1}$
provided that
$2 \le r \le N$, $0\le s\le r-1$, $r, s\in \Z$.
\end{theorem}
\begin{proof}
By Corollary~\ref{corollary:integrable}
it is sufficient to establish that
\begin{equation}\label{kappabn}
\kappa(L)<\frac12 \codim(L)+m(L)
\end{equation}
where $L$ is an arbitrary subspace from the intersection
semi-lattice ${\mathcal{L}}(B_N)$ and $m(L)=m_{K_{r-1,r-s-1}}(L)$.
Choose a subspace $L\in\LL$ of type $(\lambda, z)$.
It is easy to see that $\codim(L)=N-l(\lambda)$ and
$$\kappa(L)=\sum\frac{\lambda_i(\lambda_i-1)}{2r}+\frac{z(z-1)}{r}+
\big(\frac{1}{2}-\frac{s}{r}\big)z.$$

Since $K_{r-1, r-s-1}\subset \K_{r-1, r-s-1}$, one has
$m(L)\ge m_{\K_{r-1, r-s-1}}(L)$.

Assume that $z>0$. Let us estimate the value of
$m_{\K_{r-1, r-s-1}}(L)$. Consider a partition
$\nu=(\nu_1, \ldots, \nu_{r-1})$, a set
$$T=\{\tau_1, \ldots, \tau_{r-s-1}\}\subset\{1, \ldots, r-1\}$$
and
the corresponding polynomial $p_{\nu, T}\in\K_{r-1, r-s-1}$ introduced
in Section~\ref{section:ideals}. A polynomial $\bar{p}_{\nu, T}$
from the $\Sy_N$-orbit of $p_{\nu, T}$
gives rise to a presentation of each $\lambda_i$ and $z$ as
a sum of $r-1$ nonnegative summands
$$\lambda_i=\lambda_{i, 1}+\ldots+\lambda_{i, r-1},\quad
z=\zeta_1+\ldots+\zeta_{r-1}$$
so that
$$\nu_j=\lambda_{1, j}+\ldots+\lambda_{l(\lambda), j}+\zeta_j.$$
Moreover,
$$\mult_L(\bar{p}_{\nu})=\sum\limits_{i=1}^{l(\lambda)}\sum\limits_{j=1}^{r-1}
\frac{\lambda_{ij}(\lambda_{ij}-1)}{2}
+\sum\limits_{j=1}^{r-1}\zeta_j(\zeta_j-1)+
\sum\limits_{\tau\in T}\zeta_{\tau}.$$
Hence, in the notation of Lemma~\ref{lemma:quadratic-minimum}
and Lemma~\ref{lemma:krivoj-krokodil-ab}
one has
\begin{multline*}
m_{\K_{r-1, r-s-1}}(L)=\\
=\sum\limits_{i=1}^{l(\lambda)}
\Big(\min\limits_{\sum_j\lambda_{ij}=\lambda_i}
\sum\limits_{j=1}^{r-1}
\frac{\lambda_{ij}(\lambda_{ij}-1)}{2}\Big)+
\min\limits_{\sum_j\zeta_j=z}\Big(\sum_{j=1}^{r-1}\zeta_j(\zeta_j-1)+
\sum\limits_{\tau\in T}\zeta_{\tau}\Big)=\\
=\sum\limits_{i=1}^{l(\lambda)}\mu_{r-1, \lambda_i}+
\tilde{\mu}_{r-s-1, s, z}.
\end{multline*}
By Lemmas~\ref{lemma:quadratic-minimum} and~\ref{lemma:krivoj-krokodil-ab}
applied for $a=r-s-1$ and $b=s$
one has
$$m(L)\ge
\sum_{i=1}^{l(\lambda)}\frac{\lambda_i(\lambda_i-(r-1))}{2(r-1)}+
\frac{R_{r-1}(\lambda)}{2(r-1)}+ \frac{z^2}{r-1}-\frac{s}{r-1}z.$$
By Lemma~\ref{lemma:type-B-crocodile} applied for $k=r-1$ one has
\begin{multline*}
m(L)+\frac{1}{2}\codim(L)\ge\\
\ge
\sum\frac{\lambda_i(\lambda_i-1)}{2(r-1)}+\frac{z^2}{r-1}+
\big(\frac{1}{2}-\frac{s}{r-1}\big)z>\\
>\sum\frac{\lambda_i(\lambda_i-1)}{2(r-1)}+\frac{z^2}{r}+
\big(\frac{1}{2}-\frac{s+1}{r}\big)z\ge\\
\ge \sum\frac{\lambda_i(\lambda_i-1)}{2r}+\frac{z^2}{r}+
\big(\frac{1}{2}-\frac{s+1}{r}\big)z=\\
=\sum\frac{\lambda_i(\lambda_i-1)}{2r}+\frac{z(z-1)}{r}+
\big(\frac{1}{2}-\frac{s}{r}\big)z=\kappa(L)
\end{multline*}
as required.

Now assume that $z=0$ and estimate $m_{\K_{r-1,r-s-1}}$.
Arguing as in the proof of Proposition~\ref{proposition:A-type}, one
obtaines
\begin{multline*}
m(L)\ge m_{\K_{r-1,r-s-1}}(L)=m_{I^{\pm}_{r-1}}(L)=m_{I_{r-1}}(L)=
\sum_{i=1}^{l(\lambda)}\mu_{r-1,\lambda_i}=
\\
=\sum_{i=1}^{l(\lambda)}\frac{\lambda_i(\lambda_i-(r-1))}{2(r-1)}
+\frac{R_{r-1}(\lambda)}{2(r-1)}.
\end{multline*}
Thus the assertion in this case is implied by
Lemma~\ref{lemma:type-A-crocodile} applied for $k=r-1$.

Finally, choose a subspace $L\in\LL$ of type $\lambda^{-}$. It is easy to
see that the values of $\codim(L)$, $\kappa(L)$ and $m(L)$ are the same
as for a subspace of type $(\lambda, 0)$, which completes the proof.
\end{proof}

The second statement in type $B$ is about ideals $I_r^{\pm}$ and $J_s$ (see Section~\ref{section:ideals}).
\begin{theorem}
\label{proposition:B-type-second}
The function
$$g=\frac{f}{\prod_{i<j}^N (x_i^2-x_j^2)^{\frac{1}{r}} \prod_{i=1}^N x_i^{\frac12-\frac{s}{r}}}$$
is locally $L^2$-integrable
provided that
one of the following sets of conditions holds:
\begin{enumerate}
\item[(i)]
$f\in I_{r-1}^{\pm}$, $2 \le r \le N$, $r\in \Z$,  $s \ge r-1$, $s \in\R$;
\item[(ii)]
$f \in J_s$,  $0\le s \le N-1$, $s \in \Z$, $r\ge N+1$, $r \in \R$;
\item[(iii)]
$f \in \C[x]$, $s>N-1$, $s \in \R$, $r\ge N+1$, $r \in \R$;
\item[(iv)]
$f \in \C[x]$, $r<0, s<0, r,s \in \R$.
\end{enumerate}


\end{theorem}

\begin{proof}
By Corollary~\ref{corollary:integrable}
it is sufficient to establish that
\begin{equation}\label{kappabn2}
\kappa(L)<\frac12 \codim(L)+m(L)
\end{equation}
where $L$ is an arbitrary subspace from the intersection
semi-lattice ${\mathcal{L}}(B_N)$, that is $L$ has type $(\lambda, z)$, and $m(L)=m_{I_{r-1}^\pm}(L)$.
Recall that $\codim(L)=N-l(\lambda)$.

Let $2 \le r \le N$, $r \in \Z$ and let $s>r-1$.
We know that the
multiplicity $m(L)$ for $f \in I_{r-1}^\pm$ is given by
\begin{multline}\label{mLB}
m(L)=\sum\limits_{i=1}^{l(\lambda)}\frac{\lambda_i(\lambda_i-(r-1))}{2(r-1)}+
\frac{R_{r-1}(\lambda)}{2(r-1)}+\\ +\frac{z(z-r+1)}{r-1}+
\frac{\rho_{r-1}(z)(r-1-\rho_{r-1}(z))}{r-1}.
\end{multline}
For the multiplicity of the denominator we have
$$
\kappa(L)= \sum_{i=1}^{l(\lambda)}
\frac{\lambda_i(\lambda_i-1)}{2r}+ \frac{z(z-1)}{r}+ \frac{r-2s}{2r}
z.
$$
Thus the inequality \mref{kappabn2} follows by Lemma
\ref{lemma:type-B-crocodile-old} applied for $k=r-1$ and by Remark
\ref{remarkafterlem}, which completes the proof in case $(i)$.

Recall that for the ideal $J_s$, $s\in\Z_{+}$, one has
$m(L)=\max(0,\lfloor z-s\rfloor)$. Taking $r=N+1$ we obtain
$\kappa(L) < \frac12 (N-l(\lambda)) + \max(0,\lfloor z-s\rfloor)$ by
Lemma \ref{lemma:type-B-crocodile-old} applied for $k=N$. Therefore \beq{prop57proof}
\sum_{i=1}^{l(\lambda)} \frac{\lambda_i(\lambda_i-1)}{2r}+
\frac{z^2}{r} - \frac{(s+1)z}{r} <
\frac12\big(N-z-l(\lambda)\big)+\max(0,\lfloor z-s\rfloor). \eeq
Moreover, the inequality \mref{prop57proof} is valid for $r>N+1$
since it is valid for $r=N+1$ and its right hand side is
non-negative. Therefore the statement for case $(ii)$ is implied by
Corollary~\ref{corollary:integrable}. Note that the same argument
applies also for $s>N-1$, $s\in\R$, after replacing $J_s$ by
$\C[x]$. Indeed, in this situation one has $m(L)=\max(0,\lfloor
z-s\rfloor)=0$. This settles case $(iii)$.

The last case when $r,s <0$ is obvious.
\end{proof}


Let $\SS_c$ be the minimal non-zero submodule of the polynomial
representation of a rational Cherednik algebra. This submodule is
unique since any submodule is an ideal in $\C[x]$ (see also
\cite[Section 4.6]{ESG}). For generic $c$ the submodule $\SS_c$
coincides with $\C[x]$ however for special $c$ it becomes
non-trivial. As a corollary from the previous considerations
and by \cite[Proposition 4.12]{ESG}  we have the
following result on unitarity of the minimal submodule $\SS_c$.

\begin{theorem}\label{mainth}
\begin{enumerate}
\item
The minimal submodule $\SS_c$ for the rational Che\-red\-nik algebra
$H_c(D_N)$  is unitary if $c=1/(2r)$ where $1\le r \le N-1$, $r \in \Z$.

\item
The minimal submodule $\SS_c$ for the algebra $H_c(B_N)$ is unitary
if the parameter $c=(c_1,c_2)=(\frac{1}{r},\frac12-\frac{s}{r})$
satisfies the restrictions stated in Theorems
\ref{proposition:B-type}, \ref{proposition:B-type-second}. In particular, for $c_1=c_2$ the minimal
submodule is unitary for $c_1=1/r$ where $2\le r \le 2N$, $r \in
2\Z$, or $r > 2N$, $r \in \R$.

\end{enumerate}
\end{theorem}

Theorem \ref{mainth} in the case of constant multiplicity $c$
establishes unitarity of the simple module $\SS_c$ where $1/c$ has
to be a degree of the corresponding Coxeter group. The following
Proposition shows that this restriction is not necessary for
unitarity of the simple module.

\begin{proposition}\label{518}
Let $N\ge 3$. Then the minimal module $\SS_c$ is a unitary
representation of $H_c(B_N)$ {\lb}resp. $H_c(D_N)$\rb\  for $c=(1/3,
a)$ and $a \le 0$ {\lb}resp. $c=1/3$\rb.
\end{proposition}
\begin{proof}
It is sufficient to establish that
$$
\frac{f}{\prod_{i<j}^N (x_i^2-x_j^2)^{\frac{1}{3}}}
$$
is locally $L^2$-integrable for any $f \in I_2^\pm$.
Using Corollary \ref{corollary:integrable-equal-c}
and the previous calculation of multiplicities of $f \in I_k^\pm$ it is sufficient to establish that
$$
\sum{\lambda_i}^2-4\sum \lambda_i+3 R_2(\lambda) + 6 \rho_2(z)(2-\rho_2(z)) + 2{z^2} - 8 z + 6(N-l(\lambda)) >0
$$
where $\lambda=(\lambda_1,\ldots,\lambda_{l(\lambda)})$ is a partition of $N-z$ such that $\lambda_1 \ge 2$ if $z=0$. The last inequality follows using $R_2(\lambda)+N-z \ge 2 l(\lambda)$, $\lambda_i \ge 1$, $\sum \lambda_i = N-z$.
\end{proof}

\begin{remark}\label{remark:N-5}
Note that analogous considerations for $c=1/5$ show divergence of the integral expressing
the Gaussian inner product on the representation $I_4^\pm$ for any
$N\ge 5$ (cf. Remark~\ref{remark:reverse}).
\end{remark}

\begin{remark}
It would be interesting to see if the above ideals $I_k^\pm$, $J_k$, $K_{r,s}$
can be used to determine the composition series of the
polynomial representation for the algebras $H_c(D_N)$, $H_c(B_N)$ in some cases.
For instance in the case of the group $D_N$ and $c=1/(2m+1)$ for positive
integer $m \le (N-1)/2$ there is a natural sequence of submodules
$$0=I_{2m}^0
\subset I_{2m}^1 \subset \ldots \subset
I_{2m}^{[\frac{N}{2m+1}]} \subset \C[x],$$ where $I_{2m}^1=I_{2m}^\pm$ and the
support of the module $\C[x]/I_{2m}^s$ is stabilized by the
parabolic subgroup $A_{2m}^s$
(see \cite{F}; $I_{2m}^s$ vanishes on the $D_N$ orbit of the vanishing
set for the corresponding ideal $I^s$ in the composition series
for the polynomial representation
for $H_c(A_{N-1})$~\cite[Theorem~5.10]{ESG}). We note
that the support of $\C[x]/I_{2m}^{[\frac{N}{2m+1}]}$
then coincides with the support of the irreducible factor $L_{c}$ as it is
determined in~\cite[Theorem~3.1]{E}.
\end{remark}

\section{Some more unitarity results}
\label{section:moreunitarity}

In this section we present a few more results on the convergence of the integrals
\begin{equation}\label{Gaussian-later}
\int_{\R^N} {|f(x)|^2 e^{-\frac12 |x|^2}}{\prod_{\a \in \mathcal{R}_+}
|(\a,x)|^{-2 c}} dx
\end{equation}
where $\mathcal R \subset \R^N$ is an irreducible Coxeter root system with the Coxeter group $W$, and $c \in \R$.
In the case of convergence on the minimal submodule $\SS_c$ for the corresponding rational Cherednik algebra $H_c(W)$
this integral expresses the Gaussian inner product on $\SS_c$, the module $\SS_c$ is then unitary (see \cite{ESG}).
We will be assuming without loss of generality that the rank 
of $\mathcal R$ equals $N$.

\begin{proposition}
The minimal $H_{1/2}(W)$-module $\SS_{1/2}$ is unitary.
\end{proposition}

\begin{proof}
Let $I$ be the ideal of polynomials divisible by
$$\Delta_W(x)=\prod_{\a\in \mathcal R_+} (\a,x).$$
This ideal is an $H_{1/2}(W)$-module (see \cite{F}), therefore
$\SS_{1/2} \subset I$. But the function
$$
f^2(x) \Delta_W^{-1}(x), \quad f(x) \in I
$$
is locally integrable
as it is regular, hence the statement follows.
\end{proof}

\begin{proposition}\label{cox}
Let $h=h_W$ be the Coxeter number of the group $W$. Then
the Gaussian inner product \mref{Gaussian-later} converges on
$\SS_{1/h}$.
\end{proposition}
\begin{proof}
We need to establish that
$$
\frac{1}{h} < \frac{\frac12 \codim(L) + \mult_L(f)}{\mult_L (\Delta_W)}
$$
for any $f\in \SS_{1/h}$ and for arbitrary element $L$
from the lattice generated by the reflection hyperplanes.

We note that $\SS_{1/h}$ is contained in the $H_{1/h}(W)$-invariant
ideal consisting of polynomials vanishing at $0$. Therefore when $L=\{0\}$
it is sufficient to establish that $h \cdot \rk(\mathcal R)$ equals
the number of roots which is a well known fact. When $L\ne \{0\}$ the
inequality follows from the previous fact and the property that
$h_{W_0} < h$ where $h_{W_0}$ is the Coxeter number of any proper irreducible parabolic
subgroup $W_0 \subset W$.
\end{proof}

Proposition \ref{cox} gives another derivation of the following
known result.
\begin{corollary}\cite[Corollary 4.2]{ESG}
The minimal submodule $\SS_{1/h}$ is unitary.
\end{corollary}

The proof of Proposition \ref{cox} also provides a proof for the
following statement.

\begin{corollary}
The Gaussian inner
product \mref{Gaussian-later} converges on $\C[x]$ when $c < 1/h$.
\end{corollary}

Let $W_0$ be a proper irreducible parabolic subgroup in the irreducible Coxeter group~$W$. Then for the corresponding Coxeter numbers one
has $h_{W_0}<\nlb h_W$. The following lemma can be checked case by
case.
\begin{lemma}\label{lemd}
Let $d$ be the highest degree of a Coxeter group $W$ of type $E$, $F$ or $H$ such that
$d<h_W$. Then for any proper irreducible parabolic subgroup $W_0$ one has
$h_{W_0} < d$.
\end{lemma}

\begin{proposition}\label{next-to-h-singular-value}
Let $W$ be of type $E$, $F$ or $H$.  Let $c=1/d$
where degree $d$ is defined in Lemma \ref{lemd}. Then the Gaussian
inner product \mref{Gaussian-later} converges on the minimal submodule for $H_c(W)$ hence the
module is unitary.
\end{proposition}
\begin{proof}
We check integrability condition for $L=\{0\}$ first. We need to have
\begin{equation}\label{e824}
\frac{1}{d} < \frac{\frac12 N + \mult_0(f)}{\mult_0 (\Delta_{W})}
\end{equation}
where $f \in \SS_{c}$. Notice that $\mult_0(f) \ge 2$. Indeed, if the
multiplicity is~$0$ then $\SS_{c}$ has to coincide with $\C[x]$ which
is not the case as $c=1/d$ is a singular value for $W$. Now if the
multiplicity is 1 then $\SS_{c}$ contains homogeneous polynomials of
degree $1$ and hence the whole ideal of polynomials vanishing at $0$.
However this ideal is $H_c$-invariant only if $c=1/h_W$ which is not the case. Then since $\mult_0
(\Delta_{W})=\frac12 h_W N$, the inequality \mref{e824} reduces to
\beq{dineq}
d>\frac{N h_W}{N+4}
\eeq
which can be checked case by case.

Take now $L\ne\{0\}$ such that its stabilizer is a parabolic subgroup
$W_0=\prod_{i=1}^k W_i$ where parabolic subgroups $W_i$ are
irreducible and stabilize $L_i$ so $L=\bigcap_{i=1}^k L_i$. We have
$\codim(L) = \sum_{i=1}^k \rk (W_i)$ and
$$h_W \rk (W_i) > h_{W_i} \rk (W_i) = \mult_{L_i}(\Delta_{W_i}).$$
Therefore
\begin{multline}\label{wi}
\frac{\frac12 \codim(L) + \mult_L(f)}{\mult_L (\Delta_{W})} \ge
\frac{\codim(L)}{2 \mult_L (\Delta_{W_0})}  =\\
=
\frac{\sum_{i=1}^k
\rk(W_i)}{2\sum_{i=1}^k \mult_{L_i}(\Delta_{W_i})} \ge
\min\limits_{1\le i\le k}
\frac{1}{h_{W_i}}.
\end{multline}
Now the statement follows by Lemma \ref{lemd}.
\end{proof}

More explicitly Proposition \ref{next-to-h-singular-value}
shows that the minimal modules for $H_{1/9}(E_6)$, $H_{1/14}(E_7)$,
$H_{1/24}(E_8)$, $H_{1/8}(F_4)$, $H_{1/6}(H_3)$, $H_{1/20}(H_4)$ are unitary. A few more examples are provided by the following statement.

\begin{proposition}\label{proposition:other-exceptional-convergence}
The Gaussian inner products converge on the minimal submodules for $H_{1/8}(E_6)$, $H_{1/12}(E_7)$, $H_{1/5}(H_3)$ hence the modules are unitary.
\end{proposition}
\begin{proof}
The proof is parallel to the proof of
Proposition~\ref{next-to-h-singular-value}.
We consider the case of $H_{1/8}(E_6)$, other cases are similar.
The value $d=8$ satisfies \mref{dineq} hence there is convergence at
$L=\{0\}$.
Let now $L$ be such that $\dim(L)=1$ and a generic point on $L$
is stable under the subgroup $D_5 \subset E_6$.
Since $1/8=h_{D_5}$, the minimal module $\SS_{1/8}$ is contained in the
parabolic ideal consisting of polynomials vanishing on the $E_6$-orbit of $L$ which is a module for $H_{1/8}(E_6)$ (see \cite{F}).
Therefore $\mult_L(f)>0$ for $f \in \SS_{1/8}$ and the  inequality \mref{wi}
is strict as required. For $L$ with different
stabilizers the convergence follows from \mref{wi} straightforwardly.
\end{proof}

The next statement shows that the convergence of the Gaussian inner product
is preserved under the restriction functor $\Res_b$ defined in \cite{BE}.
 Let $L_b$ be the minimal stratum containing a point $b\in \R^N$,
and let $n$ be its codimension. Let $W_b$ be the parabolic subgroup of $W$ which stabilizes $L_b$.

\begin{proposition}\label{Res}
Assume the Gaussian inner product converges on the minimal $H_c(W)$-module
$\SS_c \subset \C[x_1,\ldots,x_N]$. Then the Gaussian inner product converges on the $H_c(W_b)$-module $\Res_b
(\SS_c)$.
\end{proposition}
\begin{proof}
Let $M$ be the affine plane orthogonal to $L_b$ such that $b \in M$.
Let $ L \subset M$ be the element of the intersection lattice of $W_b$ 
acting in~$M$. Note that $L= {\tilde L} \bigcap M$ where ${\tilde L}$ is an element of the intersection lattice for $W$ such that ${\tilde L}\supset L_b$. Let
$$
\delta(W,c)= \prod_{\a\in {\mathcal R}_+} (\a,x)^{c} , \quad \quad
\delta(W_b,c)= \prod_{\nad{\a\in {\mathcal R}_+}{(\a,L_b)=0}} (\a,x)^{c}
$$
Due to convergence of the initial Gaussian product we have
\begin{equation}\label{tildeL}
\mult_{\tilde L}\big(\delta(W,c)\big) <
\frac12 \codim(\tilde L) + m_{\SS_c}(\tilde L),
\end{equation}
where $m_{\SS_c}(\tilde L)$, as usual,
denotes the minimal multiplicity of the elements of $\SS_c$ on $\tilde L$.

The module $\Res_b(\SS_c)$ is obtained by completion $\widehat{(\SS_c)}_b$
at $b$ with subsequent extraction of the polynomial part such that
the polynomials are constant in the direction of the stratum $L_b$.
Under this process we have
$$m_{\SS_c}(\tilde L) \le m_{\Res_b \SS_c}(\tilde L) = m_{\Res_b \SS_c}(L).$$
Since $\codim(L) = \codim(\tilde L)$ and
$\mult_{\tilde L}\big(\delta(W,c)\big) = \mult_L\big(\delta(W_b,c)\big)$,
the inequality \mref{tildeL} implies
$$
\mult_{L}\big(\delta(W_b,c)\big) < \frac12 \codim(L) + m_{\Res_b \SS_c}(L),
$$
and the statement follows by Corollary \ref{corollary:integrable}.
\end{proof}
Note that the module $\Res_b
(\SS_c)$ is non-trivial for any $b \in \R^N$ so the convergence
of the Gaussian inner product
on the minimal submodule for $H_c(W_b)$ also follows. Note also that the proof of Proposition \ref{Res}
works also in the case of non-constant $W$-invariant $c$.


\section{A few negative results}
\label{section:conclusionsnegative}

In this section\footnote{This section is largely based on the
comments which P.\,Etingof and S.\,Griffeth kindly provided to us on
the preliminary version of the paper.}
we explain that the minimal submodule $\SS_c$ is not
unitary in the case of the groups $D_N$, $B_N$, $c=1/N$, $N$ is odd, and
present a few more examples when the integral \mref{Gaussian-later} diverges on
the minimal submodule (cf. Remark~\ref{remark:N-5}).

We are indebted to S.\,Griffeth for explanations leading to the following
result.
\begin{proposition}\label{prgrif}\cite{Grpc}
The minimal submodule $\SS_c$ for $H_c(D_N)$ is not unitary when $N\ge
5$ is odd and $c=1/N$.
\end{proposition}

We start with the following statement.

\begin{lemma}
Let $c=1/N$ where $N$ is odd. Then $\C[x]/\SS_c$ is a non-trivial
irreducible $H_c(D_N)$-module. Also $\SS_c \cong L_\tau$ where
$L_\tau$ is irreducible $H_c(D_N)$-module corresponding to
$D_N$-module $\tau$ given by the reflection representation of $\Sy_N
\subset D_N$. More specifically, the lowest homogeneous component of
$\SS_c$ is generated by the polynomials $x_1^2-x_i^2$ for $2\le i
\le N$.
\end{lemma}

\begin{proof}
Consider the polynomial representation $\C[x]$ for the rational
Cherednik algebra $H_c(B_N)$ where $N$ is odd, $c(e_i)=0$ and $c(e_i
\pm e_j)= 1/N$. It follows from \cite[Theorem 7.5]{Gr} that this
representation has unique non-trivial submodule. On the other hand
we know that $I_{N-1}^{\pm}$ is a submodule in $\C[x]$. Therefore
the only submodule for $H_c(B_N)$ is the minimal submodule
$\SS_c^{B_N}=I_{N-1}^{\pm}$. By Proposition \ref{genidi} the
elements in its lowest homogeneous component are linearly generated
by $x_1^2-x_i^2$ where $1 \le i \le N$.

Consider now the polynomial
representation for the algebra $H_c(D_N)$, let $M$ be a non-trivial
submodule. It is clear that the minimal degree of the homogeneous
elements in $M$ is $2$. Indeed the degree cannot be~$1$ as otherwise
$M=J_{N-1}$ which is not possible for $c=1/N$. Also the degree
cannot be bigger than $2$ as in this case there are singular
polynomials in this degree for $H_c(B_N)$-module $\SS_c^{B_N}$, so that it
is not a simple module which is a contradiction.

Since the span $\langle x_1^2-x_i^2\rangle$, $1 \le i \le N$, is
irreducible $D_N$-module it follows that the lowest homogeneous
component of $M$ coincides with the lowest homogeneous component of
$I_{N-1}^\pm$, therefore $M=I_{N-1}^\pm$.
\end{proof}

Now we prove Proposition \ref{prgrif}.
\begin{proof} Let $ f= x_2^2-x_3^2 \in \SS_c$. It
is easy to check straightforwardly that
$$
\nabla_{e_1} (x_1 f) = \lambda f,
$$
where $\lambda= \frac{4}{N}-1$. Let $(\cdot,\cdot)_{\tau}$ be the
contravariant form on $L_\tau$. We have
$$
(x_1 f, x_1 f)_{\tau} = \lambda (f,f)_{\tau}.
$$
Since $\lambda<0$ when $N\ge 5$, the module $\SS_c \cong L_\tau$ is
not unitary.
\end{proof}

Proposition \ref{prgrif} shows that in general the minimal 
$H_c(W)$-module $\SS_{c}$ is not unitary when $c=1/{d_i}$, 
with $d_i$ a degree of the Coxeter group $W$, thus providing negative answer to the Cherednik's question \cite{Chered, ESG}.
However the exceptions are rare namely the only exception for the classical root systems and constant parameter $c$ is given by $W=D_N$ with odd $N$, $c=1/N$.

As we saw in Propositions \ref{518}, 
\ref{proposition:other-exceptional-convergence} 
there are also examples when the Gaussian inner product converges on the minimal $H_c(W)$-module $\SS_c$ hence the module is unitary however $c\ne 1/d_i$ for any degree $d_i$. The examples found above are $H_{1/3}(D_N)$ with $N\ge 4$ and $H_{1/5}(H_3)$. It would be interesting to investigate when exactly $\SS_c$ is unitary.

Below we give a few more examples when Gaussian product diverges on $\SS_c$. First we present some analysis in type $B$ which is similar to the Proposition \ref{prgrif} above on type $D$.

\begin{proposition}
Consider the rational Cherednik algebra $H_c(B_N)$, $N\ge 3$ with the parameters $c(e_i\pm e_j)=1/N$, $c(e_i)=a$ such that $2 a +2j/N$ is not a positive odd number for any $0\le j \le N-1$. Suppose also that $N(2a+1)>4$. Then the minimal submodule $\SS_c$ is not unitary.
\end{proposition}

\begin{proof}
The $H_c(B_N)$ module $\C[x]$ has unique non-trivial submodule if the first stated restriction for $a$ holds \cite[Theorem 7.5]{Gr}. Hence
we have $\SS_c = L_\tau = I_{N-1}^\pm$. The direct norm calculation similar to the proof of Proposition \ref{prgrif} gives $$
(x_1 f, x_1 f)_{\tau} = \lambda (f,f)_{\tau},
$$
where $f= x_2^2-x_3^2$ and $\lambda = (4-N)/N - 2 a$. Under the second stated restriction for $a$ we have $\lambda<0$ hence the module is not unitary.
\end{proof}

As a corollary we have the following statement on non-unitarity in the case of equal parameters.

\begin{corollary}
Let $N\ge 3$ be odd, let $c=1/N$. Then the minimal submodule $\SS_c$ for $H_c(B_N)$ is not unitary.
\end{corollary}

Using Proposition \ref{Res} we get further corollary on divergence of the integral for the Gaussian product for the equal parameter cases.

\begin{corollary}
The integral \mref{Gaussian-later} is not convergent on the minimal submodules for $H_c(B_N)$ when $c=1/k$ with $3 \le k \le N$, $k$ is odd, and for  $H_{1/3}(F_4)$.
\end{corollary}

The following statement was explained to us by P.\,Etingof.

\begin{proposition}\label{EtingofProp}\cite{Etpc}
The integral \mref{Gaussian-later} is not convergent on the minimal submodule $\SS_c$
for $H_c(E_7)$ when $c=1/10$.
\end{proposition}
\begin{proof}
Let $b\in \R^7$ be a point such that its stabilizer is isomorphic to
the subgroup $E_6\subset E_7$. Note first that there are elements $p
\in \SS_c$ such that $p(b)\ne 0$. Indeed, otherwise
$b \in \supp (\C[x]/\SS_c)$
and $\Res_b (\C[x]/\SS_c)$ is a non-trivial factor of
the polynomial representation for $H_c(E_6)$. But this is not
possible since $c=1/10$ is not a singular value for $H_c(E_6)$.

Let $L$ be the one-dimensional linear space containing $b$, so
that one has
$\codim(L)=\nlb 6$. By above there are elements $p \in \SS_c$ such that
$p(b)\neq 0$ and thus
$\mult_L(p)=0$. For the convergence of the Gaussian product on $p$ we
need to have \beq{e7} \frac{1}{10}<\frac {3}{K(L)}, \eeq where
$K(L)$ is then equal to the number of positive roots in $E_6$, so
$K(L)=36$. Thus \mref{e7} fails.
\end{proof}

\begin{proposition}
The integral \mref{Gaussian-later} is not convergent on the minimal submodules
for $H_{1/9}(E_7)$, $H_{1/9}(E_8)$, $H_{1/7}(E_7)$, $H_{1/15}(E_8)$.
\end{proposition}
The proof is parallel to the proof of Proposition \ref{EtingofProp}
where one takes $L$ of codimension $6$ stabilized by the
parabolic subgroup $D_6 \subset E_7 \subset E_8$ for the first two cases.
One can take $L$ stabilized by the parabolic $D_5 \subset E_7$
for the case of $H_{1/7}(E_7)$,
and one can take $L$ stabilized by
the parabolic $E_7\subset E_8$ in the last case.

The following statement follows from the Proposition \ref{Res}
and from Propositions \ref{prgrif}, \ref{EtingofProp}.
\begin{proposition}
The integral \mref{Gaussian-later} is not convergent on the minimal submodules for $H_{1/m}(D_N)$ where $5 \le m \le N$, $m$ is odd, for $H_{1/10}(E_8)$, $H_{1/7}(E_8)$, $H_{1/5}(E_8)$, $H_{1/5}(E_7)$, $H_{1/5}(E_6)$.
\end{proposition}

\vspace{5mm}

{\bf Funding.} 
M.F. was partially supported by the 
Engineering and Physical Sciences Research Council (grant number
EP/F032889/1). 
M.F. also acknowledges support of the 
British Council (PMI2 Research Cooperation award). C.S. was
partially supported by the Russian 
Foundation for Basic Research (grant numbers 08-01-00395-a,
11-01-00185-a, 11-01-00336-a), grant N.Sh.-4713.2010.1 and by AG 
Laboratory GU-HSE, RF government grant
11 11.G34.31.0023.

\vspace{5mm}

{\bf Acknowledgements.} We are very grateful to P.\,Etingof
for attracting our attention to the problem and for helpful
discussions at different stages of the work. We are also very grateful to S.\,Griffeth for stimulating discussions and explanations, and 
to I.\,Cheltsov, J.\,Koll\'ar and Yu.\,Prokhorov for useful advice.

\end{document}